\documentclass[11pt]{article}

\usepackage{amsmath,amssymb,amsthm,mathtools,amsrefs}
\usepackage[margin=1in]{geometry}
\usepackage{enumitem}
\usepackage{lmodern}
\usepackage{xcolor}
\usepackage{hyperref}
\usepackage{comment}
\usepackage[normalem]{ulem}
\usepackage{authblk}

\colorlet{myrefcolor}{gray!55!blue}
\hypersetup{
  colorlinks=true,
  linkcolor=myrefcolor,
  citecolor=myrefcolor,
  urlcolor=myrefcolor
}

\newtheorem{theorem}{Theorem}[section]
\newtheorem{proposition}[theorem]{Proposition}
\newtheorem{lemma}[theorem]{Lemma}
\newtheorem{corollary}[theorem]{Corollary}
\theoremstyle{definition}

\theoremstyle{remark}
\newtheorem{remark}[theorem]{Remark}

\numberwithin{equation}{section}

\newcommand{\B}{\mathbb B}
\newcommand{\C}{\mathbb C}
\newcommand{\diag}{\operatorname{diag}}
\newcommand{\Herm}{\operatorname{Herm}}
\newcommand{\rank}{\operatorname{rank}}
\newcommand{\Span}{\operatorname{span}}

\title{Degree-Three Rational Sphere Maps:\\
Sharp Denominator Region and Gram Normal Forms}

\author[1]{Eden Danielsen-Jensen}
\affil[1]{{\footnotesize Department of Mathematics, Brigham Young University, Provo, UT 84602, USA}}
\author[2]{Dusty Grundmeier}
\affil[2]{{\footnotesize Department of Mathematics, The Ohio State University, 231 W 18th Ave, Columbus, OH 43210, USA}}
\author[3]{Abdullah Al Helal}
\author[2]{Valentin D.\ Kunz}
\author[3]{Ji\v{r}\'i Lebl}
\affil[3]{{\footnotesize Department of Mathematics, Oklahoma State University, Stillwater, OK 74078, USA}}
\author[4]{Ming Xiao\thanks{M.\ Xiao is supported in part by NSF grants DMS-2045104 and DMS-2554635.}}
\affil[4]{{\footnotesize Department of Mathematics, University of California San Diego,
9500 Gilman Drive, La Jolla, CA 92093, USA}}
\author[5,6]{Weixia Zhu\thanks{W.\ Zhu is supported in part by FWF grants 10.55776/ESP367 and 10.55776/PAT1879425.}}
\affil[5]{{\footnotesize School of Mathematical Sciences, Xiamen University, Xiamen 361005,
China}}
\affil[6]{{\footnotesize Faculty of Mathematics, University of Vienna, Vienna 1090, Austria}}

\affil[ ]{{\footnotesize
edendj@math.byu.edu
grundmeier.1@osu.edu
ahelal@okstate.edu
v$\_$kunz@startmail.com
lebl@okstate.edu
m3xiao@ucsd.edu
wxzhu@xmu.edu.cn
}}

\date{}

\begin{document}

\sloppy

\maketitle

\begin{abstract}
We study degree-three rational sphere maps in two complex variables. After a
standard normalization, the denominator of such a map takes the form
\[
g_\sigma(z)=1+\sigma_1 z_1^2+\sigma_2 z_2^2,
\qquad
\sigma_1,\sigma_2\geq 0.
\]
A basic question is: which pairs $(\sigma_1,\sigma_2)$ can actually occur as
the denominator of a degree-three rational sphere map? The first main result
of the paper gives a complete answer: such a denominator occurs if and only if
\[
0\leq \sigma_1,\sigma_2<1,
\qquad
\sqrt{1-\sigma_1^2}+\sqrt{1-\sigma_2^2}>1.
\]
Our approach converts the sphere-mapping condition into a finite-dimensional
Gram-matrix positivity problem.
Furthermore, for each admissible parameter
$
\sigma=(\sigma_1,\sigma_2),
$
we determine all possible minimal target dimensions in which the corresponding
denominator $g_\sigma$ can be realized. We also give a Gram-matrix normal form
for maps with a fixed denominator and compute, for each admissible $\sigma$,
the dimension of the moduli space of equivalence classes of rational sphere
maps realizing $g_\sigma$.
Finally, we extend the Gram-matrix method to arbitrary source dimension and
obtain a general sufficient condition for the existence of degree-three
rational sphere maps.

\end{abstract}

\section{Introduction}

A classical problem in several complex variables and CR geometry is to
understand proper holomorphic maps between unit balls,
\[
F\colon\B^n\longrightarrow\B^N,
\]
up to automorphisms of the source and target. Algebraic, geometric, and
analytic viewpoints meet naturally in the study of rational sphere maps.
Forstneri{\v c} \cite{F} proved that proper holomorphic maps between balls
with sufficient boundary regularity are rational, while Cima and Suffridge
\cite{CimaSuffridge} showed that such rational proper maps have no poles on
the unit sphere.
Consequently, there is a bijection between rational proper maps between balls and nonconstant rational sphere maps.
D'Angelo developed algebraic and Hermitian-form methods for polynomial and
rational sphere maps in
\cites{DAngeloPolynomial,DAngeloClassification}; see also the books
\cites{DAngeloHypersurfacesBook,DAngeloSphereMaps}. Beginning with Huang's pioneering work \cite{H1}, geometric and analytic methods developed by Huang, and later with collaborators such as Ji, Xu, and Yin, have led to far-reaching rigidity, classification, and gap results; see, for example, \cites{HJ,HJX,HJY}. The literature on proper holomorphic maps between balls
and CR mappings of spheres is extensive, and these references represent only
a small part of it; see references in \cite{DAngeloSphereMaps} for many research works on this topic.

Despite this substantial body of work, explicit classification of ball maps remains
difficult even in low degree when the codimension is unrestricted. We call two ball maps \emph{spherically equivalent} if they differ only by composing
automorphisms of the source and target balls. The
degree-one theory is classical. 
In degree two, Faran, Huang, Ji, and Zhang \cite{FHJZ} proved that every proper rational holomorphic map from $\B^2$ into $\B^N$ is spherically equivalent to a polynomial map. A complete degree-two normal-form classification in all source dimensions was later established in \cite{LeblHermitian}; the resulting normal forms turn out to be monomial maps.
In
sharp contrast, Faran, Huang, Ji, and Zhang \cite{FHJZ} also constructed
degree-three rational sphere maps that are not spherically equivalent to any
polynomial map. Thus degree three is the first degree 
at which a nonconstant denominator can be an intrinsic feature of the
spherical-equivalence class.
Write a rational sphere map in the form
\[
F=\frac{P}{g},
\]
where $P=(p_1,\ldots,p_N)^t$ is a vector-valued polynomial and $g$ is a
scalar polynomial. We say that this representation is \emph{reduced} if
$g,p_1,\ldots,p_N$ have no common nonconstant polynomial factor.
For a reduced representation, we define
\[
\deg F=\max\{\deg P,\deg g\}.
\]
The condition that $F$ maps the unit sphere to the unit sphere is encoded by
the Hermitian \emph{sphere identity}
\[
\|P(z)\|^2=|g(z)|^2,
\qquad z\in\partial\B^n.
\]
A normalized representation with $g(0)=1$ will be called \emph{pointed} if
$F(0)=0$, equivalently if $P(0)=0$.

The normal-form theorem in \cite{LeblNormalForms}*{Theorem~1.1} gives a
starting point for the degree-three problem. Every
reduced rational proper map of degree three is spherically equivalent to a
reduced pointed map $P/g$ 
whose denominator has the form
\[
g(z)=1+q(z),
\qquad
q(z)=\sum_{j=1}^n\sigma_jz_j^2,
\qquad
0\leq\sigma_j < 1.
\]
We call this the $\Lambda$-normal form. The multiset
$\{\sigma_1,\ldots,\sigma_n\}$ is a spherical invariant.   Once the parameters
are ordered, two representatives in $\Lambda$-normal form are spherically
equivalent precisely up to a target unitary and a source unitary preserving
$q$.

In two source variables, the normalized denominator therefore takes the form
\begin{equation}\label{eq:denominator}
g(z)=1+\sigma_1z_1^2+\sigma_2z_2^2,
\qquad
0\leq\sigma_1,\sigma_2 < 1.
\end{equation}
To understand degree-three rational sphere maps in two complex variables, the
first natural question is therefore: which pairs $(\sigma_1,\sigma_2)$ can
occur as the denominator of a reduced degree-three rational sphere map? Our
first main theorem gives a sharp answer.

\begin{theorem}\label{thm:main}
Let $g$ be given by \eqref{eq:denominator}. There exist an integer $N$ and a
holomorphic polynomial map $P\colon\C^2\to\C^N$, with $P(0)=0$ and
$\deg P=3$, such that $P/g$ is a reduced sphere map if and only if
\begin{equation}\label{eq:admissible-range}
0\leq\sigma_1,\sigma_2<1,
\qquad
\sqrt{1-\sigma_1^2}+\sqrt{1-\sigma_2^2}>1.
\end{equation}
\end{theorem}

Catlin and D'Angelo \cite{CD} proved that if $q$ is a polynomial satisfying $q(0)=0$ and 
$
\sup_{\overline{\B^n}}|q|<1,
$
then there exists an integer $N$ and a polynomial map $P: \C^n \to \C^N$
such that $P/(1+q):\B^n \to \B^N$ is a reduced rational proper map. For the
$\Lambda$-normalized quadratic $q(z)=\sigma_1 z_1^2+\sigma_2 z_2^2$, we have
\[
\sup_{\overline{\B^2}}|q|=\max\{\sigma_1,\sigma_2\}.
\]\
Therefore their theorem produces some reduced rational map with denominator $1+q$ whenever $\max\{\sigma_1,\sigma_2\}<1$. It does not, however, control the degree of the numerator or the target dimension. Theorem \ref{thm:main} gives the sharp degree three refinement; see the remark below.

\begin{remark}
For the particular denominator $g$ in \eqref{eq:denominator},
the pointedness requirement in Theorem~\ref{thm:main} imposes no additional
restriction on the parameters $(\sigma_1,\sigma_2)$. Indeed, suppose that
$P/g$ is a reduced degree-three sphere map, without assuming $P(0)=0$
but still assuming that the denominator is of the form \eqref{eq:denominator}.
After a target unitary transformation, write
$P(0)=(c,0,\ldots,0).$
If $P=(p_1,\ldots,p_N)^t$, comparison of the third Fourier coefficient in
the sphere identity shows that the cubic homogeneous part of $p_1$ vanishes.
Hence
\[
    \widetilde P=(z_1p_1,z_2p_1,p_2,\ldots,p_N)^t
\]
is again a reduced degree-three numerator with the same denominator and
satisfies $\widetilde P(0)=0$. Thus, for this fixed form of denominator,
the existence of a degree-three realization is equivalent to the existence
of a pointed degree-three realization.
See D'Angelo \cite{DAngeloHypersurfacesBook}*{Remark 3 on page 163} or \cite{DAngeloSphereMaps}
for more information
on this tensoring procedure.

In particular, in view of the Catlin--D'Angelo theorem  \cite{CD} above, the question
of whether a prescribed denominator $g$ of this form admits a
degree-three realization has the same answer with or without imposing
pointedness. We nevertheless retain the pointed formulation throughout,
since in the $\Lambda$-normal form the parameters $\sigma_1,\sigma_2$ are
spherical invariants, whereas for an arbitrary unpointed representation
they need not have such an invariant meaning.
\end{remark}

We call a parameter pair satisfying \eqref{eq:admissible-range}, and its
corresponding denominator, \emph{admissible}. Thus the polynomial point
$(0,0)$ is admissible. Let $\Sigma$ denote the set of admissible ordered pairs, namely those
$(\sigma_1,\sigma_2)$ satisfying \eqref{eq:admissible-range} with
$\sigma_1\geq\sigma_2$, and write
\[
g_\sigma(z)=1+\sigma_1z_1^2+\sigma_2z_2^2,
\qquad \sigma=(\sigma_1,\sigma_2)\in\Sigma.
\]

The inequalities $\sigma_j<1$ reflect the zero-free denominator property
\cite{CimaSuffridge}. The square-root inequality is the additional
restriction imposed by degree three. D'Angelo's realization theorem gives
numerators after a suitable scalar rescaling of a prescribed denominator
\cite{DAngeloRationalCR}*{Theorem~4.1}; here the denominator is fixed in
$\Lambda$-normal form and we determine exactly when a cubic numerator
exists. On the symmetric line $\sigma_1=\sigma_2=\sigma$, the condition
reduces to $\sigma<\sqrt3/2$, in agreement with the symmetric calculation
in \cite{DAngeloRationalCR}.

The proof begins with D'Angelo's Fourier identities \cite{DAngeloRationalCR}*{Corollary~3.1}. Write
$P=P_1+P_2+P_3$, where $P_j$ is vector-valued homogeneous of degree $j$. Using the
Hermitian inner product that is linear in the first variable, these
equations take the form
\begin{equation}\label{eq:dangelo-intro}
\begin{aligned}
q
&=\frac{\langle P_3,P_1\rangle}{\|z\|^2},\\
\langle P_3,P_2\rangle
&=-\|z\|^2\langle P_2,P_1\rangle,\\
\|P_3\|^2+\|P_2\otimes z\|^2
+\|P_1\otimes z^{\otimes2}\|^2
&=\|z^{\otimes3}\|^2+|q|^2\|z\|^2.
\end{aligned}
\end{equation}
We translate these equations
into the adjoint notation used for the Gram-matrix calculation in
Section~\ref{sec:gram}. For the denominator-existence problem, the middle
identity creates no additional obstruction: in the
construction we may take $P_2=0$, while in the necessity argument the term
$\|P_2\otimes z\|^2$ is absorbed into a positive-semidefinite cubic Gram term. The
first and third identities produce a positive definite $2\times2$ matrix
$X$ and a residual $4\times4$ matrix $M_X$. A degree-three numerator exists
precisely when $M_X\geq0$ and $M_X\neq0$;
see Theorem \ref{thm:gram} below.
Taking the direct sum of a numerator with its image under a sign symmetry
makes $X$ diagonal. The residual matrix then splits into two $2\times2$
blocks, and their positivity gives
\eqref{eq:admissible-range}. The same calculation gives strict feasibility:
every admissible denominator admits a choice of $X$ for which $M_X>0$.

Having determined exactly which denominators occur, a natural next question is: for a fixed admissible denominator, which minimal target dimensions can occur?
If
$P=(p_1,\ldots,p_N)^t$, define the \emph{minimal target dimension} of $P$  to be
$
\dim_{\C}\Span_{\C}\{p_1,\ldots,p_N\}.
$
In other words, it is the smallest $k$ such that a target unitary carries the
image into $\B^k\subset\B^N$.

\begin{theorem}[Allowable target dimensions]
\label{thm:attainable-targets}
Fix $\sigma\in\Sigma$ and a positive integer $N$. If
$\sigma\neq(0,0)$, then there is a reduced pointed degree-three sphere map with
$\Lambda$-normal-form denominator $g_\sigma$ and minimal target $\B^N$ if
and only if
\[
4\leq N\leq9.
\]
At the polynomial point $\sigma=(0,0)$, such a map exists if and only if
\[
3\leq N\leq9.
\]
In particular, every admissible nonconstant denominator occurs for a map
into $\B^4$, and target dimension four is optimal uniformly over the
nonconstant fibers.
\end{theorem}

Beyond existence and target dimension, we would also like to understand the
moduli space of all degree-three rational sphere maps realizing a fixed
denominator, up to spherical equivalence.
Let
$\mathcal W$ be the nine-dimensional space of scalar polynomials in
$z_1,z_2$ of degree at most three with zero constant term. We call $F=P/g$
\emph{full rank} when the components of $P$ span $\mathcal W$. Writing $P$
in a normalized monomial basis for $\mathcal W$ gives a positive
semidefinite Hermitian Gram matrix $H$. The rank of $H$ is the minimal target
dimension, thus $F$ is full rank precisely when $H>0$. For a fixed denominator, the
sphere identity cuts out a spectrahedron in a real affine space of dimension
$25$; its points parametrize target-unitary classes of numerators. For every admissible denominator, the corresponding spectrahedron contains a full-rank point.

Let $\mathcal M_\sigma$ be the moduli space of spherical-equivalence classes
of all reduced maps of exact degree three with ordered
$\Lambda$-normal-form denominator $g_\sigma$, and set
$\mathcal M=\bigsqcup_{\sigma\in\Sigma}\mathcal M_\sigma$.

\begin{theorem}\label{thm:moduli-dimension}
For fixed $\sigma\in\Sigma$,
\[
\dim_{\mathbb R}\mathcal M_\sigma=
\begin{cases}
25,&\sigma_1>\sigma_2>0,\\
24,&\sigma_1=\sigma_2>0,\\
24,&\sigma_1>0=\sigma_2,\\
21,&\sigma_1=\sigma_2=0.
\end{cases}
\]
Moreover, $\dim_{\mathbb R}\mathcal M=27$; its diagonal,
coordinate-axis, and polynomial strata have dimensions $25$, $25$, and
$21$, respectively.
\end{theorem}

Here dimension means the dimension of the open dense full-rank principal
quotient stratum. Thus the lower-rank maps included in $\mathcal M_\sigma$
do not change the dimension. The generic value $27$ consists of $25$ Gram
parameters and two denominator parameters. The dimensions drop precisely
when the denominator has additional unitary symmetry, because the residual
source group is then larger.

Finally, the same Gram calculation admits a partial extension to $n$ source
variables, which we record in Section~\ref{sec:higher-dimensional} for
comparison. For $g=1+\sum_{j=1}^n\sigma_jz_j^2$ with
$0\leq\sigma_j<1$ , we prove that
$
\sum_{j=1}^n\sqrt{1-\sigma_j^2}>n-1
$
is sufficient for a reduced pointed realization of degree three
(see Theorem~\ref{thm:higher-sufficiency}). It is sharp for $n=2$ but not
necessary for $n>2$. On the symmetric ray
$\sigma_1=\cdots=\sigma_n=\sigma$, the sharp condition is
$
\sigma<\sqrt{3/(n+2)}.
$
See Theorem~\ref{thm:symmetric-higher}. Thus the two-dimensional square-root
condition does not extend as a necessary condition in higher dimensions.

The paper is organized as follows. Section~\ref{sec:gram} derives the
finite-dimensional Gram criterion, and Section~\ref{sec:feasibility} proves
Theorem~\ref{thm:main}. Section~\ref{sec:full-rank} constructs full-rank
maps and determines all allowable minimal target dimensions.
Section~\ref{sec:spectrahedron} develops the Gram normal form and proves the
moduli-space theorem. Section~\ref{sec:higher-dimensional} concludes with
the higher-dimensional comparison.

\paragraph{Acknowledgments.}
Part of this work was carried out at the Institute for Advanced Study,
Princeton, during the Workshop on Several Complex Variables and Related
Topics. We thank the Institute for its hospitality and the Minerva Research
Foundation for its support of the workshop.

\paragraph{Use of AI tools.} The authors used various versions of OpenAI’s ChatGPT during the course of this work. The tools were used to explore and test possible proof strategies and to assist with editing the exposition. The authors carefully verified, refined, and expanded all AI output used. The authors take full responsibility for the correctness, originality, and content of the paper.

\section{The Gram-matrix formulation}\label{sec:gram}

Set
\[
r=\|z\|^2=|z_1|^2+|z_2|^2,
\qquad
q=\sigma_1z_1^2+\sigma_2z_2^2,
\qquad
g=1+q
\]
and introduce the normalized monomial columns
\begin{equation}\label{eq:monomial-columns}
V_1=\begin{pmatrix}z_1\\z_2\end{pmatrix},
\qquad
V_2 =
\begin{pmatrix}
z_1^2\\
\sqrt{2} \, z_1 z_2\\
z_2^2
\end{pmatrix},
\qquad
V_3 =
\begin{pmatrix}
z_1^3\\
\sqrt{3} \, z_1^2 z_2\\
\sqrt{3} \, z_1 z_2^2\\
z_2^3
\end{pmatrix}.
\end{equation}
Let $^{*}$ denote the conjugate transpose. We use the Hermitian inner
product that is linear in the first variable, so
$\langle U,V\rangle=V^*U$. Then $V_j^*V_j=r^j$ for $j=1,2,3$.
Multiplication by $q$ from
linear to cubic forms is represented by
\begin{equation}\label{eq:qT}
qV_1=TV_3,
\qquad
T=
\begin{pmatrix}
\sigma _1&0&\dfrac{\sigma _2}{\sqrt3}&0\\[5pt]
0&\dfrac{\sigma _1}{\sqrt3}&0&\sigma _2
\end{pmatrix}.
\end{equation}
In particular,
\begin{equation}\label{eq:T-norm}
\|qV_1\|^2=V_3^*T^*TV_3=|q|^2r.
\end{equation}
If $(\sigma _1,\sigma _2)\neq(0,0)$, then $T$ has rank two.

For a positive diagonal matrix
$
X=\diag(x,y),
$
let $S_X$ be determined by
\begin{equation}\label{eq:S-def}
V_3^* S_X V_3 = r^2 V_1^* X V_1.
\end{equation}
Thus
\begin{equation}\label{eq:S-explicit}
S_X=
\diag\left(x,\frac{2x+y}{3},\frac{x+2y}{3},y\right).
\end{equation}
Define
\begin{equation}\label{eq:MX}
M_X=I_4+T^*T-S_X-T^*X^{-1}T.
\end{equation}
Here and below, $H\geq0$ means that the Hermitian matrix $H$ is
positive semidefinite, while $H>0$ means that it is positive definite.

Since $S_{I_2}=I_4$, we have $M_{I_2}=0$. The following lemma shows that, when $q$ is nonconstant, $I_2$ is the only positive diagonal matrix $X$ for which $M_X=0$.

\begin{lemma}\label{lem:graph-point}
Assume $(\sigma _1,\sigma _2)\neq(0,0)$. If $X>0$ is diagonal and
$M_X=0$, then $X=I_2$.
\end{lemma}

\begin{proof}
Write $X=\diag(x,y)$. The $(1,1)$ and $(4,4)$ entries of
$M_X=0$ give
\[
(1-x)\left(1-\frac{\sigma _1^2}{x}\right)=0,
\qquad
(1-y)\left(1-\frac{\sigma _2^2}{y}\right)=0.
\]
If $\sigma _1\sigma _2>0$, the $(1,3)$ and $(2,4)$ entries are
\[
\frac{\sigma _1\sigma _2}{\sqrt3}\left(1-\frac1x\right)\quad
\text{and}\quad
\frac{\sigma _1\sigma _2}{\sqrt3}\left(1-\frac1y\right),
\]
respectively. so $x=y=1$. If $\sigma _1=0<\sigma _2$, the first diagonal identity
gives $x=1$, and the $(2,2)$ entry then gives $y=1$. The remaining
axis case is symmetric.
\end{proof}

We next record two elementary facts concerning reducedness and common factors, the first of which will also be used in the proof of Theorem~\ref{thm:main}.

\begin{lemma}\label{lem:zero-free-denominator}
Let $P/g$ be a reduced nonconstant rational sphere map. Then $g$ has no
zero on the closed unit ball $\overline{\B^2}$.
\end{lemma}

\begin{proof}
Suppose that an irreducible factor $h$ of $g$ has a zero in $\B^2$.
Choose a smooth point of the hypersurface $\{h=0\}$ in the ball that does
not lie on any other irreducible factor of $g$. Holomorphy of every
component of $P/g$ forces each scalar component of $P$ to vanish on a
relative neighborhood of that point in $\{h=0\}$. The identity theorem
on the irreducible hypersurface, followed by the Nullstellensatz, gives
$h\mid P_j$ for every component $P_j$, contrary to reducedness. Thus $g$
has no zero in $\B^2$. The
absence of zeros on $\partial\B^2$ is the boundary theorem of Cima and
Suffridge \cite{CimaSuffridge}.
\end{proof}

\begin{lemma}\label{lem:homogeneous-divisor}
Let $h$ be a nonconstant polynomial with $h(0)\neq0$. Then $h$
cannot divide a nonzero homogeneous polynomial.
\end{lemma}

\begin{proof}
Write $h=h_0+\cdots+h_e$ as a sum of homogeneous parts, where
$h_0\neq0$, $h_e\neq0$, and $e\geq1$. If $H=hk$ were nonzero
and homogeneous, let $k_m$ and $k_M$ be the lowest and highest nonzero
homogeneous parts of $k$. The lowest part of $hk$ is $h_0k_m$, while
the highest is $h_e k_M$. These are nonzero and have different degrees,
a contradiction.
\end{proof}

We now relate the existence of a reduced sphere map with denominator $g$ to positivity of the matrix $M_X$. The idea is to encode the linear part of the numerator by its Gram matrix $X$. After a symmetrization, $X$ may be taken to be diagonal, while the remaining homogeneous terms are encoded by the residual matrix $M_X$. This leads to the following criterion.

\begin{theorem}[Gram criterion]\label{thm:gram}
Assume
$
0\leq\sigma _1,\sigma _2<1
$ 
and 
$
(\sigma _1,\sigma _2)\neq(0,0).
$
The following are equivalent.
\begin{enumerate}[label=\textup{(\arabic*)}]
\item There is a polynomial $P$, with $P(0)=0$ and $\deg P=3$,
such that $P/g$ is a reduced sphere map.
\item There is a positive diagonal matrix $X$ such that
$M_X\geq0$ and $M_X\neq0$.
\end{enumerate}
\end{theorem}

\begin{proof}
Suppose first that $P/g$ is a reduced sphere map. We begin by making its
linear Gram matrix diagonal. Let
$
D_1=\diag(-1,1)
$
and replace $P$ by
\begin{equation}\label{eq:numerator-symmetrization}
\widehat P(z)=\frac1{\sqrt2}\bigl(P(z)\oplus P(D_1z)\bigr).
\end{equation}
Since $g(D_1z)=g(z)$, the map $\widehat P/g$ is again a sphere map.
It has the same degree as $P$ and is reduced, because any common factor of
$g$ and the components of $\widehat P$ would also divide all components of
$P$. If the linear part of the original numerator is $A_0V_1$, then the
linear coefficient matrix in \eqref{eq:numerator-symmetrization} is
$
A=\frac1{\sqrt2}
\begin{pmatrix}A_0\\ A_0D_1\end{pmatrix},
$
and hence
$
A^*A=\frac12\bigl(A_0^*A_0+D_1A_0^*A_0D_1\bigr)
$
is diagonal. We now suppress the hat and write
\[
P=P_1+P_2+P_3,
\qquad
P_1=AV_1,
\qquad
P_3=BV_3.
\]

To translate D'Angelo's equations into this notation, we repeat the
two-variable Fourier comparison from
\cite{DAngeloRationalCR}*{Theorem~3.1 and Corollary~3.1}.
For $z\in\partial\B^2$, apply the sphere identity at
$e^{i\theta}z$. Note that
\[
P_j(e^{i\theta}z)=e^{ij\theta}P_j(z),
\qquad
q(e^{i\theta}z)=e^{2i\theta}q(z).
\]
Comparison of the $e^{2i\theta}$-coefficients gives
$
P_1^*P_3=q $ on $
r=1.
$
By homogeneity, this is equivalent to
$P_1^*P_3=qr$. Using \eqref{eq:qT} and comparing bihomogeneous monomials, we obtain $A^*B=T$. Since $T$ has rank two, $A$ has rank two. Hence $X:=A^*A$
is positive definite and diagonal.
On the other hand, comparison of the $e^{i\theta}$-coefficients gives
\[
P_1^*P_2+P_2^*P_3=0
\qquad\text{on }r=1.
\]
Homogenizing the two terms to the same bidegree gives the global identity
\begin{equation}\label{eq:mixed-Fourier}
rP_1^*P_2+P_2^*P_3=0.
\end{equation}
The constant Fourier coefficient in the sphere identity gives
\[
\|P_1\|^2+\|P_2\|^2+\|P_3\|^2=1+|q|^2
\qquad\text{on }r=1.
\]
Homogenizing to bidegree $(3,3)$, we obtain
\begin{equation}\label{eq:constant-Fourier}
r^2\|P_1\|^2+r\|P_2\|^2+\|P_3\|^2
=r^3+r|q|^2.
\end{equation}
Since $V_1\otimes P_2$ is a vector of homogeneous cubics, there is a
matrix $K$ such that
$
V_1\otimes P_2=KV_3
$.
Hence
\begin{equation}\label{eq:G2}
r\|P_2\|^2
=\|V_1\otimes P_2\|^2
=V_3^*K^*KV_3.
\end{equation}
Substituting \eqref{eq:S-def}, \eqref{eq:T-norm}, and \eqref{eq:G2}
into \eqref{eq:constant-Fourier} gives
\[
B^*B+K^*K=I_4+T^*T-S_X.
\]
Set $B_0:=B-AX^{-1}T$.
Since $A^*B=T$ and $A^*A=X$, we have $A^*B_0=0$, which implies that the two
terms in $B=B_0+AX^{-1}T$ are orthogonal, and $B^*B=B_0^*B_0+T^*X^{-1}T$.
It follows that
\begin{equation}\label{eq:MX-orthogonal}
M_X=B_0^*B_0+K^*K\geq0.
\end{equation}

If $M_X=0$, both positive semidefinite terms in
\eqref{eq:MX-orthogonal} vanish. Thus $K=0$ and $B_0=0$, so $P_2=0$ and
$B=AX^{-1}T$.
Lemma~\ref{lem:graph-point} gives $X=I_2$, so we have
\[
B=AT,
\qquad
P_3=ATV_3=qP_1.
\]
Thus $P=gP_1$, contrary to reducedness. Hence $M_X\neq0$.

Conversely, we construct a numerator directly from the Gram data. Suppose that $X>0$ is diagonal, $M_X\geq0$, and
$M_X\neq0$. Define
\begin{equation}\label{eq:canonical-MX-numerator}
P_X=
\left(X^{1/2}V_1+X^{-1/2}TV_3\right)
\oplus M_X^{1/2}V_3.
\end{equation}
On $r=1$, equations \eqref{eq:S-def}, \eqref{eq:qT}, and
\eqref{eq:MX} give
\begin{align*}
\|P_X\|^2
&=V_3^*\bigl(S_X+T^*X^{-1}T+M_X\bigr)V_3+q+\bar q\\
&=V_3^*(I_4+T^*T)V_3+q+\bar q\\
&=1+|q|^2+q+\bar q
=|g|^2.
\end{align*}
Furthermore,
\[
|q(z)|\leq\max(\sigma _1,\sigma _2)\,\|z\|^2<1,
\qquad z\in\B^2,
\]
so $g$ has no zero in $\B^2$. Thus $P_X/g$ is holomorphic in the ball
and is a sphere map.

Since $M_X\neq0$, the second summand in
\eqref{eq:canonical-MX-numerator} contains a nonzero homogeneous cubic,
so $P_X$ has degree three. If a nonconstant factor $h$ of $g$ divided
every scalar component of $P_X$, projection onto the second target summand
would show that $h$ divides a nonzero homogeneous cubic. This contradicts
Lemma~\ref{lem:homogeneous-divisor}, since $h(0)\neq0$. Hence $P_X/g$ is
reduced. The construction uses at most six scalar components.
\end{proof}

\section{The sharp parameter region}
\label{sec:feasibility}

By Theorem~\ref{thm:gram}, it remains to determine for which parameter pairs $(\sigma_1, \sigma_2)$ there exists a positive diagonal matrix $X$ with $M_X \geq 0$ and $M_X \neq 0$. We now solve this positivity problem explicitly.

\begin{theorem}[Sharp range and strict feasibility]
\label{thm:sharp-feasibility}
Assume
$
0\leq\sigma _1,\sigma _2<1
$ and $
(\sigma _1,\sigma _2)\neq(0,0).
$
The following are equivalent.
\begin{enumerate}[label=\textup{(\arabic*)}]
\item There is a diagonal $X>0$ such that
$M_X\geq0$ and $M_X\neq0$.
\item $\sqrt{1-\sigma _1^2}+\sqrt{1-\sigma _2^2}>1$.
\item There is a diagonal $X>0$ such that $M_X>0$.
\end{enumerate}
\end{theorem}

\begin{proof}
 Set
\begin{equation}
\alpha:=\sqrt{1-\sigma _1^2},
\qquad
\beta:=\sqrt{1-\sigma _2^2}.  
\end{equation}
Write $X=\diag(1-s,1-t),$ where $s,t\in\mathbb R$ are parameters to be determined and $s,t<1$. With respect to the index pairs $(1,3)$ and $(2,4)$, $M_X$ is the direct sum
\[
M_X=M^{(13)}\oplus M^{(24)},
\]
where
\begin{equation}\label{eq:block13}
M^{(13)}
=
\begin{pmatrix}
s&0\\[2pt]
0&\dfrac{s+2t}{3}
\end{pmatrix}
-\frac{s}{1-s}
\begin{pmatrix}
\sigma_1^2&\dfrac{\sigma_1\sigma_2}{\sqrt3}\\[5pt]
\dfrac{\sigma_1\sigma_2}{\sqrt3}&\dfrac{\sigma_2^2}{3}
\end{pmatrix}
\end{equation}
and
\begin{equation}\label{eq:block24}
M^{(24)}
=
\begin{pmatrix}
\dfrac{2s+t}{3}&0\\[2pt]
0&t
\end{pmatrix}
-\frac{t}{1-t}
\begin{pmatrix}
\dfrac{\sigma_1^2}{3}&\dfrac{\sigma_1\sigma_2}{\sqrt3}\\[5pt]
\dfrac{\sigma_1\sigma_2}{\sqrt3}&\sigma_2^2
\end{pmatrix}.
\end{equation}
In particular, 
\[
(M_X)_{11}=\frac{s(\alpha^2-s)}{1-s},
\qquad
(M_X)_{44}=\frac{t(\beta^2-t)}{1-t}.
\]

To characterize when the two blocks are positive semidefinite, we also need to control their determinants. Factoring these determinants isolates the remaining scalar conditions:
\begin{equation}\label{eq:block-determinants}
\det M^{(13)}=\frac{s}{3(1-s)}\Phi_{13}(s,t),
\qquad
\det M^{(24)}=\frac{t}{3(1-t)}\Phi_{24}(s,t),
\end{equation}
where
\begin{align}
\Phi_{13}(s,t)
&:=s(\alpha^2+\beta^2-1-s)+2t(\alpha^2-s),
\label{eq:Phi13}\\
\Phi_{24}(s,t)
&:=t(\alpha^2+\beta^2-1-t)+2s(\beta^2-t).
\label{eq:Phi24}
\end{align}
Thus the functions $\Phi_{13}$ and $\Phi_{24}$ encode the determinant
conditions for the two blocks. We now use these formulas to characterize
positivity of $M_X$.

\medskip
\noindent\textbf{Claim.}
For $s,t<1$, one has $M_X\geq0$ if and only if
\begin{equation}\label{eq:box}
0\leq s\leq\alpha^2,
\qquad
0\leq t\leq\beta^2,
\end{equation}
and
\begin{equation}\label{eq:Phi-positive}
\Phi_{13}(s,t)\geq0,
\qquad
\Phi_{24}(s,t)\geq0.
\end{equation}

\noindent\emph{Proof of the claim.}
Suppose first that $M_X\geq0$. Since $(M_X)_{11},(M_X)_{44}\geq0$ and $s,t<1$, the formulas above give \eqref{eq:box}. Also $\det M^{(13)},\det M^{(24)}\geq0$. If $s>0$, then \eqref{eq:block-determinants} gives $\Phi_{13}(s,t)\geq0$; if $s=0$, then $\Phi_{13}(0,t)=2t\alpha^2\geq0$ by \eqref{eq:box}. Similarly, if $t>0$, then $\Phi_{24}(s,t)\geq0$, while if $t=0$, then $\Phi_{24}(s,0)=2s\beta^2\geq0$. Thus \eqref{eq:Phi-positive} holds.

Conversely, assume \eqref{eq:box} and \eqref{eq:Phi-positive}. We first show that $M^{(13)}\geq0$. If $0<s<\alpha^2$, then $(M_X)_{11}>0$, while \eqref{eq:block-determinants} gives $\det M^{(13)}\geq0$; hence $M^{(13)}\geq0$. If $s=0$, then
\[
M^{(13)}=\diag\left(0,\frac{2t}{3}\right)\geq0.
\]
Moreover, suppose $s=\alpha^2>0$. Then $(M_X)_{11}=0$, and since $\det M^{(13)}\geq0$, the off-diagonal entry of $M^{(13)}$ must vanish. Note that
\[
(M^{(13)})_{12}=-\dfrac{s\sigma_1\sigma_2}{\sqrt3(1-s)}
\]
and $0<s=\alpha^2<1$, we have $\sigma_1>0$, so $\sigma_2=0$. The remaining diagonal entry is then $(s+2t)/3\geq0$, and therefore $M^{(13)}\geq0$.

The proof that $M^{(24)}\geq0$ is analogous. If $0<t<\beta^2$, then $(M_X)_{44}>0$ and $\det M^{(24)}\geq0$, so $M^{(24)}\geq0$. If $t=0$, then
\[
M^{(24)}=\diag\left(\frac{2s}{3},0\right)\geq0.
\]
If $t=\beta^2>0$, then $(M_X)_{44}=0$, and $\det M^{(24)}\geq0$ forces its off-diagonal entry to vanish. Since
\[
(M^{(24)})_{12}=-\frac{t\sigma_1\sigma_2}{\sqrt3(1-t)}
\]
and $0<t=\beta^2<1$, we have $\sigma_2>0$, so $\sigma_1=0$. The remaining diagonal entry is $(2s+t)/3\geq0$. Thus $M^{(24)}\geq0$, and hence $M_X\geq0$. This proves the claim. \qed

\medskip

Before proving the equivalence of the three conditions, note that because $X=I_2$ exactly when $(s,t)=(0,0)$, Lemma~\ref{lem:graph-point} shows that $M_X \neq 0$ is equivalent to $(s,t) \neq (0,0)$. Thus \textup{(1)} is equivalent to the existence of
$(s,t)\neq(0,0)$ satisfying \eqref{eq:box} and
\eqref{eq:Phi-positive}.

\medskip
First of all, \textup{(3)} clearly implies \textup{(1)}. We now show that \textup{(1)} implies \textup{(2)}.
Assume \textup{(1)} and suppose, toward a contradiction, that $\alpha+\beta\leq1$. Then
$
1-\alpha^2-\beta^2\geq2\alpha\beta.
$
By the claim, there is $(s,t)\neq(0,0)$ satisfying \eqref{eq:box} and \eqref{eq:Phi-positive}. If $s,t>0$, then $\Phi_{13}(s,t),\Phi_{24}(s,t)\geq0$ give
\[
2t(\alpha^2-s)\geq s(1+s-\alpha^2-\beta^2)
\quad\text{and}\quad
2s(\beta^2-t)\geq t(1+t-\alpha^2-\beta^2).
\]
Multiplying and cancelling $st>0$, we obtain
\[
4(\alpha^2-s)(\beta^2-t)
\geq
(1+s-\alpha^2-\beta^2)(1+t-\alpha^2-\beta^2).
\]
The left-hand side is at most $4\alpha^2\beta^2$, whereas
\[
(1+s-\alpha^2-\beta^2)(1+t-\alpha^2-\beta^2)
\geq
(2\alpha\beta+s)(2\alpha\beta+t)>4\alpha^2\beta^2,
\]
a contradiction.

If $s>0$ and $t=0$, then $\Phi_{13}(s,0)\geq0$ gives $\alpha^2+\beta^2-1-s\geq0$, and hence $\alpha^2+\beta^2>1$, so $\alpha+\beta>1$, again a contradiction. The case $s=0<t$ is symmetric. Since $(s,t)\neq(0,0)$, these cases exhaust all possibilities. Therefore $\alpha+\beta>1$, proving \textup{(2)}.

\medskip
We next show that \textup{(2)} implies \textup{(3)}.
Assume $\alpha+\beta>1$ and set $d=\alpha+\beta-1>0$. Since $0<\alpha,\beta\leq1$, we have $d\leq\alpha,\beta$. Fix $0<\kappa<1$ and define
$s=\kappa\alpha d$ and
$t=\kappa\beta d$.
Then $0<s<\alpha^2$ and $0<t<\beta^2$, so the relevant pivots of both blocks are strictly positive. Moreover,
\[
\begin{aligned}
\Phi_{13}(s,t)
&=\kappa\alpha d^2
\bigl[\alpha+\beta+1-\kappa(\alpha+2\beta)\bigr],\\
\Phi_{24}(s,t)
&=\kappa\beta d^2
\bigl[\alpha+\beta+1-\kappa(2\alpha+\beta)\bigr].
\end{aligned}
\]
The bracketed factors are strictly positive because
\[
\alpha+\beta+1-\kappa(\alpha+2\beta)
=1-\beta+(1-\kappa)(\alpha+2\beta)>0
\]
and
\[
\alpha+\beta+1-\kappa(2\alpha+\beta)
=1-\alpha+(1-\kappa)(2\alpha+\beta)>0.
\]
Thus $\Phi_{13}(s,t),\Phi_{24}(s,t)>0$, so \eqref{eq:block-determinants} gives $\det M^{(13)},\det M^{(24)}>0$. Both blocks are therefore positive definite, and hence $M_X>0$. This proves \textup{(3)}.
\end{proof}

\begin{remark}\label{rem:residual-rank-one-impossible}
The same block form shows that a positive semidefinite $M_X$ has rank
$0$, $2$, $3$, or $4$, but never rank one. Indeed, if $s>0$, then
$M^{(13)}$ is a positive definite matrix minus a rank-one matrix, so it
cannot have rank zero; if $s=0$, then
$M^{(13)}=\diag(0,2t/3)$, which vanishes only when $t=0$. The same
argument applies to $M^{(24)}$. Thus away from $(s,t)=(0,0)$ both blocks
have rank one or two. This explains why the later realization lemma (Lemma \ref{lem:target-realization}) with $Q=0$ 
does not produce a target-three map; the separate exclusion of that target
still requires the classification argument below.
\end{remark}

\begin{remark}\label{rem:strict-feasibility}
The strict feasibility in Theorem~\ref{thm:sharp-feasibility} is
existential: for every admissible denominator, there is a choice of $X$
with $M_X>0$. It does not strengthen the forward implication of the Gram
criterion for an arbitrary fixed numerator. That argument produces
\[
M_X=B_0^*B_0+K^*K\geq0,
\]
and reducedness guarantees only that $M_X\neq0$.
\end{remark}

We can now combine the Gram criterion with the sharp range and strict feasibility result to prove Theorem~\ref{thm:main}. Recall that the theorem characterizes exactly which parameter pairs $(\sigma_1, \sigma_2)$ occur as denominators of reduced degree-three rational sphere maps: namely, 
\[
0 \leq \sigma_1, \sigma_2 < 1, \quad \sqrt{1-\sigma_1^2}+\sqrt{1-\sigma_2^2}>1.
\]
The necessity of the first inequalities follows from the zero-free property of a reduced denominator, while Theorems~\ref{thm:gram} and \ref{thm:sharp-feasibility} give the remaining necessity and the converse.
\begin{proof}[Proof of Theorem~\ref{thm:main}]
We first prove necessity. Suppose that a reduced sphere map $P/g$ exists.
If $\sigma_1\geq1$, then the point $(z_1,z_2)=(i/\sqrt{\sigma_1},0)$ lies in the closed unit ball and satisfies $g(z)=0$, contradicting Lemma~\ref{lem:zero-free-denominator}. The same argument applies to $\sigma_2$. Thus $0\leq\sigma _1,\sigma _2<1$.
If $(\sigma _1,\sigma _2)\neq(0,0)$, Theorem~\ref{thm:gram} gives a
positive diagonal $X$ with $M_X\geq0$ and $M_X\neq0$.
Theorem~\ref{thm:sharp-feasibility} then implies
$
\sqrt{1-\sigma _1^2}+\sqrt{1-\sigma _2^2}>1.
$
At $(\sigma _1,\sigma _2)=(0,0)$, this inequality is automatic.

Conversely, assume \eqref{eq:admissible-range}. At
$(\sigma _1,\sigma _2)=(0,0)$, one has $g=1$, and the polynomial map
$
P=V_3
$
satisfies $\|P\|^2=r^3=1$ on $\partial\B^2$. It has degree three and is automatically reduced.
For $(\sigma _1,\sigma _2)\neq(0,0)$, the strict-feasibility part of
Theorem~\ref{thm:sharp-feasibility} supplies a diagonal $X>0$ with
$M_X>0$. The converse direction of Theorem~\ref{thm:gram} then produces
the required degree-three reduced sphere map $P/g$.
\end{proof}

\section{Target dimensions and full-rank maps}\label{sec:full-rank}

Having determined exactly which denominator parameters admit degree-three
sphere maps, we now ask which target dimensions can occur for a fixed
admissible denominator.  The Gram formulation from the previous sections
is particularly useful for this question: the rank of an appropriate residual Gram matrix determines how many cubic components are required, while additional
quadratic components can be introduced by subtracting their contribution
from that residual matrix.

We proceed in three steps.  First, we formulate a realization lemma that
constructs sphere maps from this decomposition and keeps track of their
minimal target dimension.  We then use strict feasibility to construct a
full-rank map, realizing the largest possible target dimension. Finally, we vary the rank of the residual matrix and the number of quadratic components to construct the remaining attainable target dimensions, and then rule out the smaller possibilities.  The lowest target dimensions require separate consideration; in particular, the exclusion of target dimension three for nonconstant denominators use Faran's classification of proper maps from $\B^2$ to $\B^3$ \cite{Faran}.

Recall that
\begin{equation}
\mathcal W
=
\{f\in\C[z_1,z_2]: f(0)=0,\ \deg f\leq3\}. \label{e:4.1}
\end{equation}
Thus $\dim_{\C}\mathcal W=9$. Let $F=P/g$ be a reduced normalized
rational map, where $P=(p_1,\ldots,p_N)^t$, $P(0)=0$, and $g(0)=1$.
We say that $F$ has \emph{full rank} when the scalar components of $P$
span $\mathcal W$, equivalently when its minimal target dimension is $9$.

The following realization lemma will be used for all of the target-rank
constructions in this section. 

\begin{lemma}[Realization lemma]\label{lem:target-realization}
Assume $(\sigma_1,\sigma_2)$ is admissible, and let $X>0$ be diagonal.
Let $Q$ be a column of $k$ linearly independent
homogeneous quadratic polynomials; we
allow $k=0$ and $Q=0$.  Write
$V_1\otimes Q=KV_3$ and suppose that $Y=M_X-K^*K\geq0$ and $Y\neq0$.
Put $m=\rank Y$ and choose a full-row-rank matrix
$L\in\C^{m\times4}$ such that $L^*L=Y$. Then
\begin{equation}\label{eq:master-target-numerator}
P_{X,Q}
=
\left(
X^{1/2}V_1+X^{-1/2}TV_3,\,
Q,\,
LV_3
\right)
\end{equation}
defines a reduced degree-three sphere map $P_{X,Q}/g$ whose minimal
target is $\B^{2+k+m}$.
\end{lemma}

\begin{proof}
The three entries in \eqref{eq:master-target-numerator} take values in
mutually orthogonal target summands.  On $r=1$, the identities
\eqref{eq:qT}, \eqref{eq:S-def}, and
$
r\|Q\|^2=\|V_1\otimes Q\|^2=V_3^*K^*KV_3
$
give
\begin{align*}
\|P_{X,Q}\|^2
&=V_3^*
\left(S_X+T^*X^{-1}T+K^*K+Y\right)V_3+q+\bar q\\
&=V_3^*(I_4+T^*T)V_3+q+\bar q\\
&=|1+q|^2.
\end{align*}
Moreover, for $z\in\overline{\B^2}$,
\[
|q(z)|
\leq \max(\sigma_1,\sigma_2)\|z\|^2
\leq \max(\sigma_1,\sigma_2)<1.
\]
Thus $g$ has no zero on the closed ball, and $P_{X,Q}/g$ is a sphere
map.

Since $Y\neq0$, at least one component of $LV_3$ is a nonzero
homogeneous cubic.  Hence $P_{X,Q}$ has degree three.  If a nonconstant
factor $h$ of a nonconstant $g$ divided every component of $P_{X,Q}$,
projection onto the last target summand would show that $h$ divides every
component of $LV_3$. This contradicts
Lemma~\ref{lem:homogeneous-divisor}, because $h(0)\neq0$ and at least one
of those components is a nonzero homogeneous cubic. Thus the map is
reduced; when $g=1$, reducedness is automatic.

Finally, consider a linear relation among the scalar components of
$P_{X,Q}$.  Its degree-one part is a relation among the two components of
$X^{1/2}V_1$, so the first two coefficients vanish.  Its degree-two part
is then a relation among the components of $Q$, so the next $k$
coefficients vanish.  What remains is a relation among the components of
$LV_3$; since $L$ has full row rank, its $m$ scalar components are
linearly independent.  Thus all $2+k+m$ components are linearly
independent, proving the assertion about the minimal target.
\end{proof}

Thus the realization lemma reduces the target-dimension problem to controlling the number of quadratic components and the rank of the residual matrix:
\[
N=2+k+\rank(M_X-K^*K).
\]
We will use these two quantities independently. Setting $Q=0$ allows us to vary the target dimension through the rank of $M_X$, while choosing nonzero $Q$ allows us to add quadratic components without changing the rank of the remaining cubic contribution.

We first realize the largest possible target.  Because every scalar component
of a normalized degree-three numerator lies in the nine-dimensional space
$\mathcal W$, no minimal target dimension larger than nine can occur.
We now show that this upper bound is attained for every admissible
denominator.

\begin{corollary}\label{cor:full-rank}
Let
$g=1+\sigma _1z_1^2+\sigma _2z_2^2$.
Assume
\[
0\leq\sigma _1,\sigma _2<1,
\qquad
\sqrt{1-\sigma _1^2}+\sqrt{1-\sigma _2^2}>1.
\]
Then there is a reduced normalized sphere map $P/g$ of degree three and
full rank with target $\B^9$.
\end{corollary}

\begin{proof}
First suppose that $(\sigma _1,\sigma _2)\neq(0,0)$. By
Theorem~\ref{thm:sharp-feasibility}, there is a positive diagonal
matrix $X$ such that $M_X>0$. Choose
\[
0<\epsilon^2<\lambda_{\min}(M_X)
\]
and define
\begin{equation}\label{eq:full-rank-numerator}
P_\epsilon=
\left(
X^{1/2}V_1+X^{-1/2}TV_3,\,
\epsilon V_2,\,
(M_X-\epsilon^2I_4)^{1/2}V_3
\right).
\end{equation}
The nine scalar components of $P_\epsilon$ are recorded explicitly in
Subsection~\ref{sec:scalar-numerator}.
For $Q=\epsilon V_2$, the matrix $K$ in
$V_1\otimes Q=KV_3$ satisfies $K^*K=\epsilon^2I_4$: indeed,
$\|V_1\otimes Q\|^2=\epsilon^2r^3
=V_3^*(\epsilon^2I_4)V_3$, and comparison of bihomogeneous
coefficients gives the matrix identity.  Thus the residual matrix in
Lemma~\ref{lem:target-realization} is
$M_X-\epsilon^2I_4>0$.  The three components of $Q$ are linearly
independent, so the lemma gives a reduced degree-three map with minimal
target dimension $2+3+4=9$.  It is therefore full rank.

At $(\sigma _1,\sigma _2)=(0,0)$, choose $a,b,c>0$ with
$a+b+c=1$, and set
\[
P=(\sqrt a\,V_1,\sqrt b\,V_2,\sqrt c\,V_3).
\]
Then
\[
\|P\|^2=ar+br^2+cr^3=1
\qquad\text{on }r=1.
\]
Its scalar components form a basis of $\mathcal W$, so this polynomial
sphere map also has full rank.
\end{proof}

\subsection{Scalar form of the full-rank numerator}
\label{sec:scalar-numerator}
Having exhibited the full-rank construction explicitly, we now return to the target-dimension problem.
We expand the compact construction \eqref{eq:full-rank-numerator}. Let
\[
X=\diag(x,y)>0,
\qquad
0<\epsilon^2<\lambda_{\min}(M_X),
\]
be the choices made in the proof of
Corollary~\ref{cor:full-rank}, and write
\[
L=(M_X-\epsilon^2I_4)^{1/2}=(\ell_{jk})_{j,k=1}^4.
\]
The matrix $L$ is positive definite. Since $X$ is diagonal, the
matrices $M_X$ and $L$ preserve the coordinate subspaces indexed by
$(1,3)$ and $(2,4)$. Thus
\[
L=
\begin{pmatrix}
\ell_{11}&0&\ell_{13}&0\\
0&\ell_{22}&0&\ell_{24}\\
\ell_{13}&0&\ell_{33}&0\\
0&\ell_{24}&0&\ell_{44}
\end{pmatrix}.
\]
Because $TV_3=qV_1$, the first two components are
\[
\begin{aligned}
p_1 &=
\sqrt{x}\,z_1
+\frac{\sigma_1}{\sqrt{x}}z_1^3
+\frac{\sigma_2}{\sqrt{x}}z_1z_2^2,\\
p_2 &=
\sqrt{y}\,z_2
+\frac{\sigma_1}{\sqrt{y}}z_1^2z_2
+\frac{\sigma_2}{\sqrt{y}}z_2^3.
\end{aligned}
\]
The quadratic components are
\[
p_3=\epsilon z_1^2,\qquad
p_4=\epsilon\sqrt2\,z_1z_2,\qquad
p_5=\epsilon z_2^2,
\]
and the remaining cubic components are
\[
\begin{aligned}
p_6 &= \ell_{11}z_1^3+\sqrt3\,\ell_{13}z_1z_2^2,
&
p_7 &= \sqrt3\,\ell_{22}z_1^2z_2+\ell_{24}z_2^3,\\
p_8 &= \ell_{13}z_1^3+\sqrt3\,\ell_{33}z_1z_2^2,
&
p_9 &= \sqrt3\,\ell_{24}z_1^2z_2+\ell_{44}z_2^3.
\end{aligned}
\]
The entries of $L$ are explicit. Indeed, the two nonzero blocks of $L$ are the
positive square roots of
\[
Y^{(13)}=M^{(13)}-\epsilon^2I_2,
\qquad
Y^{(24)}=M^{(24)}-\epsilon^2I_2,
\]
where $M^{(13)}$ and $M^{(24)}$ are given in
\eqref{eq:block13}--\eqref{eq:block24}. For any positive definite
$2\times2$ matrix $Y$,
\[
Y^{1/2}
=
\frac{Y+\sqrt{\det Y}\,I_2}
{\sqrt{\operatorname{tr}Y+2\sqrt{\det Y}}}.
\]
The resulting map is
\[
F=\frac{(p_1,\ldots,p_9)^t}
{1+\sigma_1z_1^2+\sigma_2z_2^2}.
\]

If $(\sigma_1,\sigma_2)=(0,0)$, choose
$a,b,c>0$ with $a+b+c=1$. The corresponding scalar components are
\begin{align*}
p_1&=\sqrt a\,z_1,
&p_2&=\sqrt a\,z_2,\\
p_3&=\sqrt b\,z_1^2,
&p_4&=\sqrt{2b}\,z_1z_2,
&p_5&=\sqrt b\,z_2^2,\\
p_6&=\sqrt c\,z_1^3,
&p_7&=\sqrt{3c}\,z_1^2z_2,
&p_8&=\sqrt{3c}\,z_1z_2^2,
&p_9&=\sqrt c\,z_2^3.
\end{align*}

\subsection{Allowable target dimensions}
\label{sec:attainable-targets}

We now determine all possible minimal target dimensions. From Lemma~\ref{lem:target-realization}, the realization formula
\[
N=2+k+\rank Y
\]
gives two ways to vary $N$: change the rank of the residual cubic Gram matrix, or introduce additional independent quadratic components through $Q$.

To realize the smaller targets, we need finer control over the rank of $M_X$, not only its positivity. Because $M_X$ splits into two $2 \times 2$ blocks, each consisting of a positive diagonal matrix minus a rank-one contribution, the following result allows us to determine precisely when a block is positive definite and when its rank drops.
\begin{lemma}[Rank-one perturbation]\label{lem:rank-one-perturbation}
Let $\Delta>0$ be a Hermitian $m\times m$ matrix, let $c>0$, and let
$v\neq0$. Then
\[
\Delta-cvv^*\geq0
\quad\Longleftrightarrow\quad
c\,v^*\Delta^{-1}v\leq1.
\]
The matrix on the left is positive definite when the inequality is strict
and has rank $m-1$ when equality holds.
\end{lemma}

\begin{proof}
Put $w=\Delta^{-1/2}v$. Then
\[
\Delta-cvv^*
=
\Delta^{1/2}(I_m-cww^*)\Delta^{1/2}.
\]
The matrix $I_m-cww^*$ is the identity on $w^\perp$ and has eigenvalue
$1-c\|w\|^2$ on the span of $w$. Since
$\|w\|^2=v^*\Delta^{-1}v$, the result follows.
\end{proof}

Recall from Lemma~\ref{lem:target-realization}, when $Q=0$, the target dimension is simply
\[
N=2+\rank M_X.
\]
Remark~\ref{rem:residual-rank-one-impossible} shows that a nonzero positive
semidefinite $M_X$ can have rank $2$, $3$, or $4$, but never rank one.
The next lemma shows that all three of these allowable ranks actually
occur for every admissible nonzero denominator.

\begin{lemma}\label{lem:residual-ranks}
For every admissible nonzero pair $(\sigma_1,\sigma_2)$ and every
$m\in\{2,3,4\}$, there is a diagonal matrix $X>0$ such that $M_X\geq0$ and  $\rank M_X=m$.
\end{lemma}

\begin{proof}
Put
\[
\alpha=\sqrt{1-\sigma_1^2},
\qquad
\beta=\sqrt{1-\sigma_2^2},
\qquad
\delta=1-\alpha^2-\beta^2.
\]
Thus $\alpha,\beta>0$ and $\alpha+\beta>1$. Write
\[
X=\diag(1-s,1-t),
\qquad
s=c,\quad t=c\tau,
\qquad c,\tau>0.
\]
Applying Lemma~\ref{lem:rank-one-perturbation} to
\eqref{eq:block13}--\eqref{eq:block24} shows, provided $c<1$ and
$c\tau<1$, that
\[
M^{(13)}\geq0
\quad\Longleftrightarrow\quad
c\leq c_1(\tau),
\]
and
\[
M^{(24)}\geq0
\quad\Longleftrightarrow\quad
c\leq c_2(\tau),
\]
where
\[
c_1(\tau)
=
\alpha^2-\frac{\sigma_2^2}{1+2\tau}
=
\frac{2\alpha^2\tau-\delta}{1+2\tau}
\]
and
\[
c_2(\tau)
=
\frac{\beta^2}{\tau}-\frac{\sigma_1^2}{\tau+2}
=
\frac{2\beta^2-\tau\delta}{\tau(\tau+2)}.
\]

Thus $c_1(\tau)$ and $c_2(\tau)$ are precisely the critical values of
$c$ at which the $(1,3)$ and $(2,4)$ blocks, respectively, drop from
rank two to rank one.
These inequalities also include the box constraints from
Theorem~\ref{thm:sharp-feasibility}, since
\[
c_1(\tau)\leq\alpha^2,
\qquad
\tau c_2(\tau)\leq\beta^2.
\]
The perturbing vectors in both blocks are nonzero. Thus equality gives a
block of rank one, while strict inequality gives a block of rank two. The
required bounds $c<1$ and $c\tau<1$ will be verified for the choices
below.

If $\delta>0$, both functions are positive on
$
\frac{\delta}{2\alpha^2}
<\tau<
\frac{2\beta^2}{\delta}.
$
This interval is nonempty because $\alpha+\beta>1$. As $\tau$ approaches the left endpoint, $c_1\to0<c_2$, whereas at the
right endpoint $c_2\to0<c_1$. Hence there is an
interior point $\tau_0$ such that
\[
c_1(\tau_0)=c_2(\tau_0)>0.
\]
If $\delta\leq0$, both functions are positive for $\tau>0$,
\[
c_2(\tau)-c_1(\tau)\longrightarrow+\infty
\quad\text{as }\tau\longrightarrow0^+,
\]
and
\[
c_1(\tau)-c_2(\tau)\longrightarrow\alpha^2>0
\quad\text{as }\tau\longrightarrow+\infty.
\]
The same conclusion follows.

Choose $c=c_1(\tau_0)=c_2(\tau_0)$. Both blocks then have rank one, so $\rank M_X=2$. Moreover, $X>0$, since
\[
c=\alpha^2-\frac{\sigma_2^2}{1+2\tau_0}<1,\qquad
c\tau_0=\beta^2-\frac{\tau_0\sigma_1^2}{\tau_0+2}<1,
\]
where the inequalities are strict because $(\sigma_1,\sigma_2)\neq(0,0)$. To obtain rank three, choose $\tau>0$ on the side of the crossing where $0<c_1(\tau)<c_2(\tau)$ and set $c=c_1(\tau)$. The first block has rank one and the second rank two, so $\rank M_X=3$. The same estimates show that $X>0$. For rank four, Theorem~\ref{thm:sharp-feasibility} gives a diagonal $X>0$ with $M_X>0$, hence $\rank M_X=4$.
\end{proof}

We can now prove Theorem~\ref{thm:attainable-targets} and determine the complete set of minimal target dimensions for a reduced pointed degree-three sphere map with
$\Lambda$-normal-form denominator $g_\sigma$. For a nonconstant admissible denominator, the constructions follow directly from the realization formula. Taking $Q=0$ and varying $\rank M_X$ gives target dimensions 4, 5 and 6. Starting from a strictly feasible $M_X>0$, we may add one or two sufficiently small quadratic components while preserving a rank-four residual matrix, giving targets 7 and 8; the full-rank construction gives target 9.

It remains to show that no other targets occur. Dimension 9 is the absolute upper bound because all numerator components lie in $\mathcal{W}$. Targets 1 and 2 are excluded by general rigidity results for proper ball maps. Target 3 is more subtle: the Gram construction alone does not rule it out, so for a nonconstant denominator we appeal to Faran's classification of proper maps from $\B^2$ to $\B^3$. At the polynomial point, target 3 does occur and can be realized explicitly.

\begin{proof}[Proof of Theorem~\ref{thm:attainable-targets}]
Suppose first that $(\sigma_1,\sigma_2)\neq(0,0)$.  Taking $Q=0$ in
Lemma~\ref{lem:target-realization} and using
Lemma~\ref{lem:residual-ranks} gives maps with minimal target dimensions
\[
4,\quad5,\quad6.
\]

Choose a diagonal $X>0$ for which $M_X>0$. For $\ell=1,2$, let $Q_\ell$
be the column consisting of the first $\ell$ entries of $V_2$. Since
$V_1\otimes Q_\ell$ is a vector of homogeneous cubics, there is a matrix
$K_\ell$ such that
\[
V_1\otimes Q_\ell=K_\ell V_3.
\]
Put
$
G_\ell=K_\ell^*K_\ell\geq0.
$
Then
\[
r\|Q_\ell\|^2=V_3^*G_\ell V_3.
\]
For sufficiently small $\epsilon>0$,
\[
Y_\ell=M_X-\epsilon^2G_\ell>0.
\]
Indeed,
\[
V_1\otimes(\epsilon Q_\ell)=(\epsilon K_\ell)V_3,
\qquad
Y_\ell=M_X-(\epsilon K_\ell)^*(\epsilon K_\ell).
\]
Apply Lemma~\ref{lem:target-realization} with $Q=\epsilon Q_\ell$,
$K=\epsilon K_\ell$, and residual matrix $Y_\ell$. Its hypotheses hold
because the components of $Q_\ell$ are distinct monomials and
$\rank Y_\ell=4$. The resulting minimal target dimension is
$2+\ell+4=6+\ell$.  Taking $\ell=1,2$ gives dimensions $7$ and $8$,
while Corollary~\ref{cor:full-rank} gives dimension $9$.
No minimal target dimension larger than nine can occur, since every scalar
component belongs to the nine-dimensional space $\mathcal W$.

There is no proper holomorphic map from $\B^2$ to $\B^1$. Indeed, a
nonconstant proper map would have a regular fiber that is a
positive-dimensional compact complex submanifold of the Stein domain
$\B^2$, which is impossible; a constant map is not proper.  On the other
hand, every proper self-map of $\B^2$ is an automorphism by Alexander's
theorem \cite{Alexander}.  Thus a degree-three map cannot have minimal
target dimension one or two.

Suppose a map with $(\sigma_1,\sigma_2)\neq(0,0)$ had minimal target
dimension three.  Put
$
s=\max(\sigma_1,\sigma_2)<1
$
and choose $\eta>0$ such that $s(1+\eta)^2<1$.  If
$\|z\|\leq1+\eta$, then
\[
|q(z)|\leq s\|z\|^2<1,
\]
so the denominator is zero-free on a neighborhood of
$\overline{\B^2}$. Hence the map is holomorphic past the sphere, and the
sphere identity makes it proper. Faran's classification \cite{Faran}
therefore applies. Among Faran's four equivalence classes, the only one of
algebraic degree three is represented by
\[
(z_1^3,\sqrt3\,z_1z_2,z_2^3).
\]
This representative fixes the origin and has reduced denominator $1$, so
it is already in $\Lambda$-normal form with both denominator parameters
equal to zero.

Let $(\widetilde g,\widetilde P)$ be the degree-three homogenization of $(g,P)$. Since $(g,P)$ is reduced, any common homogeneous factor of $\widetilde g$ and $\widetilde P$ must be a power of $z_0$; this is impossible because some component of $P$ has a nonzero cubic term. Thus the homogenized lift is reduced of degree three. Pre- and postcomposition by automorphisms preserve common factors, and hence preserve the reduced degree. The assumed map would therefore be equivalent to the displayed representative. By the invariance of the denominator parameters \cite{LeblNormalForms}*{Theorem~1.1}, this would force $(\sigma_1,\sigma_2)=(0,0)$, a contradiction. Thus the nonconstant case gives exactly $4,\ldots,9$.

It remains to treat case
$(\sigma_1,\sigma_2)=(0,0)$, where $g=1$ and $T=0$, where $T$ is defined in \eqref{eq:qT}. For target dimension
three, take
\[
\Psi_3(z)=\left(z_1^3,\sqrt3\,z_1z_2,z_2^3\right).
\]
If $x=|z_1|^2$, $y=|z_2|^2$, and $x+y=1$, then
\[
\|\Psi_3(z)\|^2
=x^3+3xy+y^3
=x^3+3xy(x+y)+y^3
=(x+y)^3=1.
\]
Its three components are linearly independent and its degree is three.
For target dimension four, take $\Psi_4=V_3$. Then
$
\|\Psi_4(z)\|^2=r^3=1
$
on the sphere, and its four components are linearly independent.

The remaining dimensions follow from
Lemma~\ref{lem:target-realization}, and we sketch a proof below. First note that, since $T=0,$ we have
$
M_X=I_4-S_X.
$
To achieve target dimension five, we take
$X_5:=\diag(1/2,1)$, then
\[
M_{X_5}=\diag(1/2,1/3,1/6,0).
\]
Taking $Q=0$, Lemma~\ref{lem:target-realization} yields a degree-three map with minimal target
dimension $2+\rank M_{X_5}=5$. Next take $X_6:=I_2/2$, for which
$
M_{X_6}:=I_4/2.
$
Taking $Q=0,$ Lemma~\ref{lem:target-realization} yields minimal target dimension $2+\rank M_{X_6}=6$.

We next prove the case of minimal target dimensions seven, eight, and nine.  For $\ell=1,2$, let $Q_\ell$ be the column consisting of the first
$\ell$ entries of $V_2$, and write
\[
V_1\otimes Q_\ell=K_\ell V_3,
\qquad
G_\ell=K_\ell^*K_\ell.
\]
Note that
\[
V_1\otimes(\epsilon Q_\ell)=(\epsilon K_\ell)V_3,
\qquad
Y_\ell=M_{X_6}-(\epsilon K_\ell)^*(\epsilon K_\ell).
\]
Choose $\epsilon>0$ sufficiently small, one can let
$
Y_\ell=\frac12I_4-\epsilon^2G_\ell>0.
$
Applying Lemma~\ref{lem:target-realization} with
$X=X_6$, $Q=\epsilon Q_\ell$, and $
K=\epsilon K_\ell$
gives the minimal target dimension
\[
2+\ell+\rank Y_\ell=6+\ell.
\]
Thus $\ell=1,2$ gives target dimension seven and eight, while the polynomial-point
construction in Corollary~\ref{cor:full-rank} gives target dimension nine. All these
maps are reduced because $g=1$, and every construction has a nonzero
cubic component. The target-dimension-one and target-dimension-two exclusions above also
apply at the polynomial point, completing the proof.
\end{proof}

We have now completely determined the target dimensions for each admissible denominator: every nonconstant admissible denominator realizes precisely the dimensions $4,\ldots,9$, while the polynomial point also realizes dimension 3. We now pass from the existence of such realizations to the structure of the full space of numerators with a fixed denominator.

\section{Gram normal forms and spectrahedra}\label{sec:spectrahedron}

Having determined the admissible denominators and the possible target dimensions, we now turn from existence to classification. For a fixed admissible denominator $g$, we seek to describe the space of all degree-three rational sphere maps with denominator $g$, first up to unitary transformations of the target and then up to full spherical equivalence.

In this section, we describe all degree-three numerators associated with a fixed admissible denominator. We encode numerators by their Gram matrices, which turns the sphere identity into affine-linear constraints together with positivity. The Gram normal form, Theorem~\ref{thm:gram-normal-form}, identifies this Gram data with target unitary classes and shows that its rank records the minimal target dimension. We then compute the dimension of this resulting Gram spectrahedron in Proposition~\ref{prop:gram-spectrahedron}, and finally quotient by the residual source symmetries to obtain a Gram normal form corresponding to spherical-equivalence classes (Corollary \ref{cor:spherical-normal-form}) and the moduli-space dimensions in Theorem~\ref{thm:moduli-dimension}.

Let $\Herm_9$ denote the real vector space of $9\times9$ Hermitian matrices, and let $\Herm_9^{\ge0}$ be its cone of positive semi-definite matrices. Recall the denominator $g=g_\sigma$ and the admissible parameter set $\Sigma$
from the introduction. In particular, the definition of $\Sigma$ already includes the
ordering $\sigma_1\geq\sigma_2$.
Fix $\sigma\in\Sigma$ and write $g=g_\sigma$. Reuse the normalized monomial
columns from \eqref{eq:monomial-columns} and set
\[
\mathbf V=
\begin{pmatrix}V_1\\V_2\\V_3\end{pmatrix}.
\]
Thus the entries of $\mathbf V$ are a basis of $\mathcal W$ (see \eqref{e:4.1}). Every normalized numerator of degree
at most three can be written as
$
P=C\mathbf V.
$
The squared norm of $P(z)$ is 
\[
\|P(z)\|^2=\mathbf V(z)^*H\mathbf V(z),
\qquad H=C^*C\geq0.
\]

Define  the affine equality space 
\begin{equation}\label{eq:affine-equality-space}
\mathcal A_g
=
\left\{
H\in\Herm_9:
\mathbf V^*H\mathbf V=|g|^2
\text{ on }\partial\B^2
\right\},
\end{equation}
and its positive-semidefinite part
\begin{equation}\label{eq:gram-spectrahedron}
\mathcal S_g
=
\mathcal A_g\cap\Herm_9^{\geq0}.
\end{equation}
We call $\mathcal S_g$ the \emph{Gram spectrahedron} associated with
$g$.
The full-rank, positive-definite locus  of $\mathcal S_g$ will be denoted by
\[
\mathcal S_g^{\mathrm{fr}}
=
\mathcal S_g\cap\Herm_9^{>0}
=
\mathcal A_g\cap\Herm_9^{>0}.
\]

\subsection{The Gram normal form}

We call two numerator representations \emph{stably target-unitarily
equivalent} if, after adjoining zero target components as needed, their
numerators differ by a unitary transformation of the target. For $H\geq0$,
let $H^{1/2}$ denote its unique positive-semidefinite Hermitian square root.

The next theorem makes the role of the Gram matrix $H$ precise. Every feasible Gram matrix produces a canonical numerator via its positive square root, and conversely every numerator determines a unique Gram matrix. Moreover, basic geometric features of the map, including its minimal target dimension, degree, and reducedness, can all be read directly from $H$.

\begin{theorem}[Gram normal form]\label{thm:gram-normal-form}
Assume $\sigma\in\Sigma$. For $H\in\mathcal S_g$, put
\begin{equation}\label{eq:gram-normal-form}
P_H=H^{1/2}\mathbf V,
\qquad
F_H=\frac{P_H}{g}.
\end{equation}
Then the following statements hold.
\begin{enumerate}[label=\textup{(\arabic*)}]
\item $F_H$ is a pointed sphere map into $\B^9$, although the displayed
representation need not be reduced.
\item The assignment $H\mapsto[F_H]$, where brackets denote the stable
target-unitary class, is a bijection from $\mathcal S_g$ to the stable
target-unitary equivalence classes of pointed representations $P/g$, with
$\deg P\leq3$, that satisfy the sphere identity. The smallest target
dimension of the class is $\rank H$.
\item Define
$
\mathcal E_H
=\{v^*\mathbf V:v\in\operatorname{Range}H\}.
$
This is the span of the scalar numerator components for every factorization
$H=C^*C$. The representation is reduced precisely when no nonconstant
factor of $g$ divides every polynomial in $\mathcal E_H$.
\item Let $H_{c}$ be the trailing $4\times4$ principal block
of $H$, indexed by the cubic monomials. The numerator has degree three
precisely when $H_{c}\neq0$.
\item The representation $F_H$ has full rank precisely when
$\rank H=9$, equivalently when $H>0$. In that case reducedness and degree
three are automatic.
\end{enumerate}
\end{theorem}

\begin{proof}
Let $P$ be a pointed numerator of degree at most three. Relative to
$\mathbf V$, it has a unique coefficient matrix $C$:
$
P=C\mathbf V.
$
The sphere identity is equivalent to
\[
\mathbf V^*C^*C\mathbf V=|g|^2
\qquad\text{on }\partial\B^2.
\]
Thus $H=C^*C$ belongs to $\mathcal S_g$. Conversely, $H^{1/2}$ is a
factor of $H$, so $P_H$ satisfies the sphere identity. Because
$\sigma\in\Sigma$,
\[
|\sigma_1z_1^2+\sigma_2z_2^2|
\leq\max(\sigma_1,\sigma_2)\|z\|^2<1,
\qquad z\in\B^2,
\]
and hence $g$ is zero-free in $\B^2$. Therefore $F_H$ is holomorphic
and proves \textup{(1)}.

Suppose $C^*C=D^*D$. The rule
\[
Cx\longmapsto Dx
\]
is well-defined and isometric from $\operatorname{Range}C$ to
$\operatorname{Range}D$, because both squared norms and inner products
are determined by the common matrix $C^*C=D^*D$. After adjoining zero
target coordinates, this isometry extends to a target unitary. The converse
is immediate. Moreover, for any factorisation $H= C^* C$,
\begin{align}
\operatorname{Range}C^*
=
\operatorname{Range}(C^*C)
=
\operatorname{Range}H, \label{e:5.4}
\end{align}
which proves \textup{(2)}.

Moreover,
the row space of $C$, interpreted as polynomials via $\mathbf V$, is, according to \eqref{e:5.4}  exactly $\mathcal E_H$. This proves \textup{(3)}.

If
$C=[C_1\ C_2\ C_3]$ is divided into its linear, quadratic, and cubic
column blocks, then
$
H_{c}=C_3^*C_3.
$
Thus $C_3\neq0$ exactly when $H_{c}\neq0$, proving
\textup{(4)}.

Note that $g(0)=1$, while every member of $\mathcal E_H$ has zero
constant term, and
$
\dim_{\C}\mathcal E_H=\rank H.
$
Thus full rank is equivalent to $\rank H=9$, which for $H\geq0$ is
equivalent to $H>0$. If $H>0$, then $\mathcal E_H=\mathcal W$; since this space
contains both $z_1$ and $z_2$, its elements have no common nonconstant
factor, and it contains nonzero cubics. This completes the proof.
\end{proof}

The theorem allows us to isolate within the Gram spectrahedron $\mathcal{S}_g$ exactly the points corresponding to the maps of interest in this paper: reduced maps of exact degree three. Indeed, each such map belongs to the space $\mathcal S_g^{\mathrm{red},3}$ which is defined as follows.
\begin{equation}\label{eq:reduced-degree-three-locus}
\mathcal S_g^{\mathrm{red},3}
=
\left\{
H\in\mathcal S_g:
H_{c}\neq0\ \text{and no nonconstant factor of $g$ divides
every member of $\mathcal E_H$}
\right\}.
\end{equation}
The set $\mathcal S_g^{\mathrm{red},3}$ is precisely the Gram locus of
reduced maps of exact degree three with denominator $g$. Its importance stems from how the Gram locus is linked to the moduli space $\mathcal{M}$ defined in the introduction, as we shall see within Section \ref{sec.5.3}.

\subsection{Dimension of the Gram spectrahedron}

Let
$
\mathbf V_{\leq2}=
\begin{pmatrix}V_1\\V_2\end{pmatrix}
$
be the column vector whose entries are the normalized monomial basis for the five-dimensional space 
\[
\mathcal W_2
=
\{f\in\C[z_1,z_2]:f(0)=0,\ \deg f\leq2\}.
\]

To compute the dimension of the moduli space for a fixed denominator $g$, we first determine the number of free parameters. Since $\mathcal{A}_g$ is an affine space, this reduces to describing the Hermitian perturbations of a Gram matrix that do not change the sphere identity. Equivalently, we must characterize the Hermitian polynomials of degree at most three that vanish on the unit sphere. The next lemma does exactly this and yields the 25-dimensional Gram parameter space used below.

\begin{lemma}\label{lem:gram-kernel}
For $K\in\Herm_9$, the Hermitian polynomial $\mathbf V^*K\mathbf V$
vanishes on $r=1$ if and only if there is a unique $R\in\Herm_5$ such that
\begin{equation}\label{eq:kernel-factorization}
\mathbf V^*K\mathbf V=(1-r)\mathbf V_{\leq2}^*R\mathbf V_{\leq2}.
\end{equation}
Consequently, the kernel of restriction to the sphere is naturally
isomorphic to $\Herm_5$.
\end{lemma}

\begin{proof}
The vanishing ideal of the real unit sphere is generated by $1-r$.
Indeed, write $z_1=x_1+ix_2$ and $z_2=x_3+ix_4$. Division in $x_4$ by
$r-1$ leaves a remainder $ax_4+b$; evaluation at the two values
$x_4=\pm\sqrt{1-x_1^2-x_2^2-x_3^2}$ shows on an open set that $a=b=0$.
Thus $\mathbf V^*K\mathbf V=(1-r)Q$ for a polynomial $Q$. The numerator
and $1-r$ are real-valued, so $Q$ is real-valued wherever $r\neq1$, and
hence everywhere by polynomial continuation. Thus $Q$ is Hermitian.

Comparison of the least and greatest holomorphic and antiholomorphic
degrees shows that every bihomogeneous component of $Q$ has holomorphic
and antiholomorphic degrees in $\{1,2\}$. Hence
$Q=\mathbf V_{\leq2}^*R\mathbf V_{\leq2}$ for a unique $R\in\Herm_5$.
Conversely,
the right side of \eqref{eq:kernel-factorization} has degrees between one
and three and therefore determines a unique $K\in\Herm_9$. Uniqueness in
both statements follows because the products of the normalized monomials
in $\mathbf V_{\leq2}$, and likewise in $\mathbf V$, are linearly
independent.
\end{proof}

The previous lemma identifies the vector space of Hermitian perturbations that preserve the sphere identity with $\Herm_5$, so the affine space $\mathcal{A}_g$ of Gram matrices satisfying the sphere identity has real dimension 25. The full-rank construction from the previous section gives a positive-definite Gram matrix in $\mathcal{A}_g$. The next proposition combines these facts to show that the Gram spectrahedron has non-empty relative interior in this 25-dimensional affine space and that its positive-definite points parametrize, up to target-unitary equivalence, the full-rank sphere maps with denominator $g$.

\begin{proposition}\label{prop:gram-spectrahedron}
For every $\sigma\in\Sigma$, with $g=g_\sigma$, the affine space
$\mathcal A_g$ has real dimension $25$, and $\mathcal S_g^{\mathrm{fr}}$ is
nonempty and relatively open in
$\mathcal A_g$. Modulo unitary
transformations of the target, its points parametrize precisely the
normalized sphere maps with denominator $g$ and full rank in
the minimal target $\B^9$, or, stably, in a larger ball after adjoining
zero coordinates.
\end{proposition}

\begin{proof}
Let
$
\mathcal L(K)
=\left.\mathbf V^*K\mathbf V\right|_{r=1}.
$
The translation space of $\mathcal A_g$ is $\ker\mathcal L$. By
Lemma~\ref{lem:gram-kernel}, we have
\begin{equation}\label{eq:kernel-isomorphism}
\ker\mathcal L\cong\Herm_5.
\end{equation}
Therefore
\begin{equation}\label{eq:fixed-denominator-dimension}
\dim_{\mathbb R}\ker\mathcal L
=\dim_{\mathbb R}\Herm_5
=25.
\end{equation}
The affine equality space is nonempty, since the graph numerator
$
P_{*}=gV_1
$
satisfies
\[
\|P_{*}\|^2=|g|^2
\qquad\text{on }r=1.
\]
Let $H_*$ be the corresponding matrix. Then
$\mathcal A_g=H_{*}+\ker\mathcal L,$
and $\dim_{\mathbb R}\mathcal A_g=25$.

Under the stated parameter condition,
Corollary~\ref{cor:full-rank} supplies a full-rank
numerator. Its coefficient matrix has rank $9$, and hence its Gram matrix
is positive definite. Thus $\mathcal S_g^{\mathrm{fr}}\neq\varnothing$.
Positive definiteness is open, so this locus is relatively open in
$\mathcal A_g$. More explicitly, if $H_0\in\mathcal S_g^{\mathrm{fr}}$, then
$H_0+K>0$ for every sufficiently small $K\in\ker\mathcal L$. Hence
$\mathcal S_g$ contains a relative neighborhood of $H_0$, and
$
\operatorname{aff}(\mathcal S_g)=\mathcal A_g.
$
The remaining assertions now follow from
Theorem~\ref{thm:gram-normal-form}: positive definiteness is equivalent to
full rank, and $H$ determines exactly one target-unitary class.
\end{proof}

So far the denominator $g_\sigma$ has been fixed. We now allow $\sigma$ to vary over the admissible parameter region and assemble the fixed-denominator spectrahedra into a single parameter space.

The graph Gram matrix depends polynomially on
$\sigma=(\sigma _1,\sigma _2)$, and
\[
\mathcal A_{g_\sigma}
=
H_{*}(\sigma)+\ker\mathcal L.
\]
The kernel is independent of $\sigma$. Thus, upon defining
\begin{align*}
\mathfrak S
:=\{(\sigma,H):\sigma\in\Sigma,\
H\in\mathcal S_{g_\sigma}^{\mathrm{red},3}\}, \quad \text{and} \quad
\mathfrak S^{\mathrm{fr}}
:=\{(\sigma,H):\sigma\in\Sigma,\ H\in\mathcal S_{g_\sigma}^{\mathrm{fr}}\},
\end{align*}
we consequently find that $\mathfrak S$ is the total Gram parameter space for reduced maps of exact
degree three, whereas $\mathfrak S^{\mathrm{fr}}$ is its full-rank locus. The latter is open and
dense: for $H\in\mathcal S_{g_\sigma}^{\mathrm{red},3}$ and
$H_0\in\mathcal S_{g_\sigma}^{\mathrm{fr}}$, the matrix
$H_t=(1-t)H+tH_0$ is positive definite for $0<t\leq1$ and tends to $H$.
It therefore has generic real dimension
$
2+25=27.
$

\subsection{Spherical equivalence and the residual source action} \label{sec.5.3}

The Gram normal form has already quotiented out unitary transformations of the target. To pass from target-unitary equivalence to full spherical equivalence, it remains to account for the source transformations that preserve the normalized denominator. Because the denominator parameters are fixed and ordered, only a residual subgroup of $U(2)$ remains.

Let
\[
D_\sigma=\diag(\sigma_1,\sigma_2),
\qquad
G_\sigma
=
\{U\in U(2):U^tD_\sigma U=D_\sigma\}.
\]
Thus $G_\sigma$ is the residual unitary source group preserving
$g_\sigma$. Define the source representation $\rho$ by
\begin{equation}\label{eq:source-representation}
\mathbf V(Uz)=\rho(U)\mathbf V(z).
\end{equation}
Precomposition by $U\in G_\sigma$ induces the action
\begin{equation}\label{eq:source-action-on-Gram}
H\longmapsto \rho(U)^*H\rho(U)
\end{equation}
on $\mathcal S_{g_\sigma}$. This action preserves positivity, rank,
reducedness, and degree.

On $\mathfrak S$ declare
$(\sigma,H)\sim(\tau,K)$ if $\sigma=\tau$ and
$K=\rho(U)^*H\rho(U)$ for some $U\in G_\sigma$. Set
$\mathcal M=\mathfrak S/\mathord{\sim}$, and let
$\mathcal M_\sigma$ be its fiber over $\sigma$.
The image of $\mathfrak S^{\mathrm{fr}}$ is denoted by
$\mathcal M^{\mathrm{fr}}$, and its fiber over $\sigma$ by
$\mathcal M_\sigma^{\mathrm{fr}}$.

\begin{lemma}\label{lem:residual-source-groups}
For ordered parameters $1>\sigma_1\geq\sigma_2\geq0$,
\[
G_\sigma=
\begin{cases}
\{\diag(\epsilon_1,\epsilon_2):\epsilon_j\in\{1,-1\}\},
&\sigma_1>\sigma_2>0,\\[3pt]
O(2),&\sigma_1=\sigma_2>0,\\[3pt]
\{\diag(\epsilon,e^{i\theta}):
\epsilon\in\{1,-1\},\ \theta\in\mathbb R\},
&\sigma_1>0=\sigma_2,\\[3pt]
U(2),&\sigma_1=\sigma_2=0.
\end{cases}
\]
\end{lemma}

\begin{proof}
Since $U$ is unitary, $(U^t)^{-1}=\overline{U}$, where the bar denotes
entrywise complex conjugation. Multiplying $U^tD_\sigma U=D_\sigma$ on
the left by $(U^t)^{-1}$ gives
\[
D_\sigma U=\overline{U}D_\sigma.
\]
Conjugating gives $D_\sigma\overline{U}=UD_\sigma$, and combining these
identities yields $D_\sigma^2U=UD_\sigma^2$.
If $\sigma_1>\sigma_2>0$, the second identity forces $U$ to be
diagonal, and the first then forces both diagonal entries to be $1$ or
$-1$. If $\sigma_1=\sigma_2>0$, the defining identity becomes
$U^tU=I_2$; together with unitarity, this says that $U$ is real
orthogonal. When $\sigma_1>0=\sigma_2$, the matrix $D_\sigma^2$ has
distinct eigenvalues, so $U$ is diagonal; the defining identity then says
that the first coordinate can only change sign, while the second can
acquire an arbitrary phase. At the
polynomial point the quadratic form vanishes, so every unitary is allowed.
\end{proof}

Geometrically this result says the following. The more symmetric the denominator, the larger the residual source symmetry group.

We can now combine the two reductions. The Gram matrix removes target-unitary freedom, while quotienting by $G_\sigma$ removes the remaining source-unitary freedom. Thus spherical equivalence classes are obtained by taking the Gram spectrahedron for each admissible denominator and quotienting by its residual source symmetry group.

\begin{corollary}[Spherical Gram normal form]
\label{cor:spherical-normal-form}
After passing to minimal-target representatives, the spherical-equivalence
classes of reduced degree-three rational sphere maps satisfy
\begin{equation}\label{eq:spherical-moduli}
\mathcal M\cong
\bigsqcup_{\sigma\in\Sigma}
\mathcal S_{g_\sigma}^{\mathrm{red},3}/G_\sigma.
\end{equation}
More explicitly, two Gram normal forms with ordered parameters
$(\sigma,H)$ and $(\tau,K)$ are spherically equivalent if and only if
\[
\sigma=\tau
\qquad\text{and}\qquad
K=\rho(U)^*H\rho(U)
\]
for some $U\in G_\sigma$. The minimal target dimension of the class is
$\rank H$. By Theorem~\ref{thm:attainable-targets}, the possible ranks are
$4,\ldots,9$ when $\sigma\neq(0,0)$ and $3,\ldots,9$ at the polynomial
point. For every $\sigma\in\Sigma$,
\begin{equation}\label{eq:fixed-moduli-quotients}
\mathcal M_\sigma
\cong
\mathcal S_{g_\sigma}^{\mathrm{red},3}/G_\sigma,
\qquad
\mathcal M_\sigma^{\mathrm{fr}}
\cong
\mathcal S_{g_\sigma}^{\mathrm{fr}}/G_\sigma.
\end{equation}
\end{corollary}

\begin{proof}
The normal-form theorem in \cite{LeblNormalForms}*{Theorem~1.1} puts every
reduced degree-three map into the pointed form used here, with normalized
denominator $g_\sigma$. The multiset
$\{\sigma_1,\sigma_2\}$ is a spherical invariant and is ordered by
convention in $\Sigma$. Two maps in this form are spherically equivalent
precisely when their numerators differ by a target unitary and a source
unitary in $G_\sigma$.

Theorem~\ref{thm:main} and the ordering of the parameters force
$\sigma\in\Sigma$, while
Theorem~\ref{thm:gram-normal-form} identifies the target-unitary class with
the unique Gram matrix $H$. Under source precomposition, that matrix
changes according to \eqref{eq:source-action-on-Gram}. Conversely, every
such source transformation and target unitary gives a spherical
equivalence. This proves \eqref{eq:spherical-moduli} and
the first identification in \eqref{eq:fixed-moduli-quotients}; restricting
to the positive-definite locus gives the second.
\end{proof}

We now prove Theorem~\ref{thm:moduli-dimension} and determine the dimension of the moduli space  $\mathcal{M}_{\sigma}$ for every choice of $(\sigma_1, \sigma_2) \in \Sigma$. To prove the theorem, we combine the preceding results. In particular, for fixed $\sigma$, we proved that the full-rank Gram space has dimension 25, so it remains to quotient by the source group $G_\sigma$ and verify that the generic stabilizer is trivial.

\begin{proof}[Proof of Theorem~\ref{thm:moduli-dimension}]
By \eqref{eq:fixed-moduli-quotients}, the full-rank locus is
$
\mathcal M_\sigma^{\mathrm{fr}}
\cong
\mathcal S_{g_\sigma}^{\mathrm{fr}}/G_\sigma.
$
The density argument above, together with openness of the quotient map,
shows that this is open and dense in $\mathcal M_\sigma$. It therefore
suffices to compute its principal quotient dimension.
Proposition~\ref{prop:gram-spectrahedron} gives
$
\dim_{\mathbb R}\mathcal S_{g_\sigma}^{\mathrm{fr}}=25
$
for every admissible fixed $\sigma$, including the coordinate axes and
the polynomial point.  Lemma~\ref{lem:residual-source-groups} gives
\[
\dim_{\mathbb R}G_\sigma=0,1,1,4
\]
in the four cases of Theorem~\ref{thm:moduli-dimension}. It remains to show
that the principal stabilizer is trivial. Choose
$H_0\in\mathcal S_{g_\sigma}^{\mathrm{fr}}$ and average over the compact
group $G_\sigma$:
\[
\overline H
=
\int_{G_\sigma}\rho(U)^*H_0\rho(U)\,dU.
\]
Because $g_\sigma(Uz)=g_\sigma(z)$, every conjugate
$\rho(U)^*H_0\rho(U)$ lies in $\mathcal A_{g_\sigma}$, and it remains
positive definite. The affine equality space $\mathcal{A}_{g_{\sigma}}$ and its positive-definite
locus are convex, so $\overline H\in\mathcal S_{g_\sigma}^{\mathrm{fr}}$.
Haar invariance shows that $\overline H$ is $G_\sigma$-invariant.

Under the identification $\ker\mathcal L\cong\Herm_5$ in
\eqref{eq:kernel-isomorphism}, the source action is equivariant. More
precisely, if
$
\mathbf V_{\leq2}(Uz)=\rho_{\leq2}(U)\mathbf V_{\leq2}(z),
$
then the action on $\Herm_5$ is
\[
R\longmapsto
\rho_{\leq2}(U)^*R\rho_{\leq2}(U).
\]
This follows immediately by applying the source change to
\eqref{eq:kernel-factorization} and using $r(Uz)=r(z)$. Now use the
decomposition
\[
\mathcal W_2
=
\operatorname{span}_{\mathbb C}\{z_1,z_2\}
\mathbin\oplus
\operatorname{span}_{\mathbb C}\{z_1^2,z_1z_2,z_2^2\}
\]
and consider the Hermitian form
\[
R_0=
\begin{pmatrix}
1&0&1&0&0\\
0&2&0&0&1\\
1&0&0&0&0\\
0&0&0&0&0\\
0&1&0&0&0
\end{pmatrix}.
\]
If $U\in U(2)$ stabilizes $R_0$, then its upper-left block gives
$U^*\diag(1,2)U=\diag(1,2)$, so
$
U=\diag(e^{i\alpha},e^{i\beta}).
$
The two off-diagonal entries coupling $z_1$ to $z_1^2$ and $z_2$
to $z_2^2$ then force $e^{i\alpha}=e^{i\beta}=1$. Thus $R_0$ has
trivial stabilizer in $U(2)$, and hence in every subgroup $G_\sigma$.

Let $K_0\in\ker\mathcal L$ correspond to $R_0$.  For sufficiently small
nonzero real $t$,
$
\overline H+tK_0\in\mathcal S_{g_\sigma}^{\mathrm{fr}}.
$
Since $\overline H$ is $G_\sigma$-invariant, the stabilizer of this
matrix is contained in the stabilizer of $K_0$ and is therefore trivial.
The space $\mathcal S_{g_\sigma}^{\mathrm{fr}}$ is convex, hence connected.
The principal orbit-type theorem for compact group actions now shows that
the trivial-stabilizer locus is relatively open and dense in each fiber.
The slice theorem then gives
\[
\dim_{\mathbb R}
\left(\mathcal S_{g_\sigma}^{\mathrm{fr}}/G_\sigma\right)
=
25-\dim_{\mathbb R}G_\sigma
\]
on the principal quotient stratum \cite{Bredon}.
Because $\mathcal M_\sigma^{\mathrm{fr}}$ is open and dense in
$\mathcal M_\sigma$, these are also the dimensions of
$\mathcal M_\sigma$.

It remains only to count denominator parameters.  The distinct-positive
chamber is two-dimensional, each nonzero diagonal or coordinate-axis
stratum is one-dimensional, and the polynomial point is zero-dimensional.
Adding these dimensions to the fixed-denominator quotient dimensions gives
$27,25,25,21$, respectively.
\end{proof}

\section{A higher-dimensional comparison}\label{sec:higher-dimensional}

Having obtained a sharp characterization in two complex variables, it is
natural to ask how much of the Gram-matrix approach extends to higher
dimensions. In this section only, let $n\geq2$, put $r=\|z\|^2$, and write
\[
q(z)=\sum_{j=1}^n\sigma_jz_j^2,
\qquad
g=1+q.
\]
We first show that the two-variable construction extends to
give a natural sufficient condition for the existence of a reduced
degree-three sphere map. We then restrict to the symmetric case
$\sigma_1=\cdots=\sigma_n$ and obtain the sharp bound there. Comparing
the two conditions shows that, unlike when $n=2$, the natural
higher-dimensional sufficient condition is not necessary when $n>2$.

For a multi-index $\nu$, set $\nu!=\nu_1!\cdots\nu_n!$ and define the
normalized monomial columns
\[
\mathcal V_k=
\left(\sqrt{\frac{k!}{\nu!}}\,z^\nu\right)_{|\nu|=k},
\qquad
d_k=\binom{n+k-1}{k}.
\]
Thus $\mathcal V_k^*\mathcal V_k=r^k$. Let $X$ be an $n \times n$ diagonal matrix.  Define $T_n$ and $S^{(n)}_X$ by 
\[
q\mathcal V_1=T_n\mathcal V_3,
\qquad
\mathcal V_3^*S^{(n)}_X\mathcal V_3
=r^2\mathcal V_1^*X\mathcal V_1,
\]
and, for $X>0$, put
\begin{equation}\label{eq:higher-residual}
M_X^{(n)}
=I_{d_3}+T_n^*T_n-S^{(n)}_X-T_n^*X^{-1}T_n.
\end{equation}

The Gram construction from Section~\ref{sec:gram} does not depend
essentially on having two variables. The same argument shows that once a
positive diagonal matrix $X$ can be found for which the residual matrix
$M_X^{(n)}$ is positive definite, a reduced pointed degree-three sphere map can be
constructed directly. We record this observation first so that the
existence problem reduces to finding such an $X$.

\begin{lemma}[Strict Gram realization]\label{lem:higher-strict-gram}
Suppose $0\leq\sigma_j<1$ for $1\leq j\leq n$. If some diagonal $X>0$  satisfies
$M_X^{(n)}>0$, then $g$ is the denominator of a reduced pointed
degree-three sphere map. The target dimension may be taken to be
$n+d_3$.
\end{lemma}

\begin{proof}
Take mutually orthogonal target summands of dimensions $n$ and $d_3$, and set
\[
P=\left(
X^{1/2}\mathcal V_1+X^{-1/2}T_n\mathcal V_3,
\ (M_X^{(n)})^{1/2}\mathcal V_3
\right).
\]
On $r=1$, the definitions of $T_n$, $S^{(n)}_X$, and $M_X^{(n)}$ give
\[
\begin{aligned}
\|P\|^2
&=\mathcal V_3^*
\left(S^{(n)}_X+T_n^*X^{-1}T_n+M_X^{(n)}\right)\mathcal V_3
+q+\bar q\\
&=\mathcal V_3^*(I_{d_3}+T_n^*T_n)\mathcal V_3+q+\bar q\\
&=1+|q|^2+q+\bar q
=|g|^2.
\end{aligned}
\]
The parameter bounds imply
$
|q(z)|\leq\max_j\sigma_j\,\|z\|^2<1
$ on $ z\in\overline{\B^n}$.
So $g$ has no zero on the closed ball.
Because $M_X^{(n)}>0$, its square root is invertible and the second target
summand has degree three. If a nonconstant factor $h$ of $g$ divided all
components of $P$, it would divide every component of
$(M_X^{(n)})^{1/2}\mathcal V_3$, hence every cubic monomial. This is
impossible by Lemma~\ref{lem:homogeneous-divisor}, because $h(0)\neq0$.
Thus $P/g$ is reduced and has degree three.
\end{proof}

Motivated by the sharp two-variable condition
$\sqrt{1-\sigma_1^2} + \sqrt{1-\sigma_2^2}>1$, we consider its natural higher-dimensional
analogue
$\sum_{j=1}^n\sqrt{1-\sigma_j^2}>n-1.$
Although this condition will no longer be necessary when $n>2$, it is
strong enough to guarantee strict positivity of the residual Gram matrix.
By Lemma~\ref{lem:higher-strict-gram}, it is enough to find a diagonal $X>0$ for which $M_X^{(n)}>0$. We choose $X$ explicitly from the parameters $\alpha_j$ and verify positivity by decomposing the residual matrix into elementary blocks.

\begin{theorem}[Higher-dimensional sufficiency]
\label{thm:higher-sufficiency}
Let $n\geq2$, $0\leq\sigma_j<1$ for $1 \leq j \leq n$, and set
$
\alpha_j=\sqrt{1-\sigma_j^2}.
$
Assume that
\begin{equation}\label{eq:higher-sufficient-condition}
\sum_{j=1}^n\alpha_j>n-1,
\end{equation}
then there is a reduced pointed degree-three sphere map with denominator
$g=1+\sum_{j=1}^n\sigma_jz_j^2$.
\end{theorem}

\begin{proof} Since $\sigma_j <1$, we have
$\alpha_j>0$. Now, put
$
\delta=\sum_{j=1}^n\alpha_j-(n-1)>0,
$
and set
\[
s_j=\frac12\alpha_j\delta,
\qquad
X=\diag(1-s_1,\ldots,1-s_n).
\]
For every $j,$ since the sum of the other $n-1$ numbers $\alpha_k$ is at most $n-1$,
$\delta\leq\alpha_j$. Therefore
$
0<s_j<\alpha_j^2\leq1,
$
and hence $X>0$.

We now show that $M_X^{(n)}>0$. For fixed $i$, let $E_i$ be the coordinate
subspace in the normalized cubic space spanned by
\[
z_i^3,
\qquad
\sqrt3\,z_i z_j^2\quad(j\neq i).
\]
These subspaces are mutually orthogonal. The remaining coordinate vectors
are the square-free cubics $\sqrt6\,z_i z_jz_k$, where $i<j<k$.
The matrix $M_X^{(n)}$ preserves this decomposition. On $E_i$ we have
\begin{equation}\label{eq:higher-block}
M_X^{(n)}\big|_{E_i}
=\Delta_i-\frac{s_i}{1-s_i}v_iv_i^*,
\end{equation}
where
\[
(\Delta_i)_{jj}=
\begin{cases}
s_i,&j=i,\\[2pt]
\dfrac{s_i+2s_j}{3},&j\neq i,
\end{cases}
\quad \text{ and } \quad
(v_i)_j=
\begin{cases}
\sigma_i,&j=i,\\[2pt]
\dfrac{\sigma_j}{\sqrt3},&j\neq i.
\end{cases}
\]
On the coordinate vector $\sqrt6\,z_i z_jz_k$, the corresponding entry is
$
(s_i+s_j+s_k)/3>0.
$

Each $\Delta_i$ is positive definite. Lemma~\ref{lem:rank-one-perturbation}
applied to \eqref{eq:higher-block} shows that this block is positive definite
exactly when
\begin{equation}\label{eq:higher-block-condition}
s_i\left(
1+\sum_{j\neq i}\frac{\sigma_j^2}{s_i+2s_j}
\right)<\alpha_i^2.
\end{equation}
Condition \eqref{eq:higher-sufficient-condition} also gives
$
\alpha_i+\alpha_j>1
$ for $i\neq j$,
since the remaining $n-2$ terms have sum at most $n-2$. Hence
\[
\frac{1-\alpha_j^2}{\alpha_i+2\alpha_j}
\leq1-\alpha_j.
\]
Using $\sigma_j^2=1-\alpha_j^2$ and the definition of $s_j$, we obtain
\[
s_i\left(
1+\sum_{j\neq i}\frac{\sigma_j^2}{s_i+2s_j}
\right)
=\alpha_i\left(
\frac{\delta}{2}+
\sum_{j\neq i}\frac{1-\alpha_j^2}{\alpha_i+2\alpha_j}
\right)\leq\alpha_i\left(
\frac{\delta}{2}+\sum_{j\neq i}(1-\alpha_j)
\right)=\alpha_i\left(\alpha_i-\frac{\delta}{2}\right)
<\alpha_i^2.
\]
Thus every block in \eqref{eq:higher-block} is positive definite, and so is
$M_X^{(n)}$. Lemma~\ref{lem:higher-strict-gram} completes the proof.
\end{proof}

The preceding condition is not necessary in higher dimensions. As we will now see, the symmetric
ray already gives the sharp comparison.
This is due to the fact that, in the symmetric case, the additional permutation and sign symmetries allow us to sharpen the preceding sufficient condition to one that is also necessary. Averaging over these symmetries reduces the linear Gram matrix to a scalar matrix, after which positivity of the residual matrix can be checked explicitly.

\begin{theorem}[Sharp symmetric bound]\label{thm:symmetric-higher}
Let $n\geq2$ and $\sigma\geq0$. There is a reduced pointed degree-three
sphere map with denominator
$
g=1+\sigma\sum_{j=1}^nz_j^2
$
if and only if
\begin{equation}\label{eq:symmetric-sharp-bound}
\sigma<\sqrt{\frac3{n+2}}.
\end{equation}
\end{theorem}

\begin{proof}
Suppose first that a reduced map $\frac{P}{g}$ exists. If $\sigma\geq1$, then
\[
z_1=i\sqrt{\frac{\sigma+1}{2\sigma}},
\qquad
z_2=\sqrt{\frac{\sigma-1}{2\sigma}},
\qquad
z_3=\cdots=z_n=0
\]
is a sphere point at which $g=0$, contradicting the zero-free
property of reduced denominators on $\overline{\B^n}$
\cite{CimaSuffridge}. Thus $\sigma<1$. The case
$\sigma=0$ already satisfies
\eqref{eq:symmetric-sharp-bound}, so assume $\sigma>0$.

Write $
P=P_1+P_2+P_3,$ where $P_j$ is vector-valued homogeneous of degree $j$. Write
$
P_1=A\mathcal V_1$, $
P_3=B\mathcal V_3,
$
and put $X=A^*A$. The same Fourier-coefficient comparison used in the proof
of Theorem~\ref{thm:gram} gives
$
A^*B=T_n.
$
Since $\sigma>0$, the polynomials $qz_1,\ldots,qz_n$ are linearly
independent. Hence $T_n$ has rank $n$, and so does $A$. Therefore $X>0$.

Let $\mathcal G_n$ be the finite group of signed permutation matrices and
symmetrize the numerator by setting
\[
\widehat P(z)=\frac1{\sqrt{|\mathcal G_n|}}
\bigoplus_{U\in\mathcal G_n}P(Uz).
\]
Since $g(Uz)=g(z)$, the quotient $\widehat P/g$ is again a pointed
degree-three sphere map. It is reduced because the direct sum contains the
original numerator as the summand corresponding to $U=I_n$.

Similarly as for $P,$ write $
\widehat P= \widehat P_1+ \widehat P_2+\widehat P_3,$
and $\widehat P_1=\widehat A\mathcal V_1$. The matrix $\widehat A$ is the
vertical stack of the matrices $AU/\sqrt{|\mathcal G_n|}$, $U \in \mathcal{G}_n$, so its Gram
matrix is
\[
\widehat X
=\frac1{|\mathcal G_n|}\sum_{U\in\mathcal G_n}U^*XU
=\mu I_n,
\qquad
\mu=\frac{\operatorname{tr}X}{n}>0.
\]
Write $\widehat P_3=\widehat B\mathcal V_3$, decompose
\[
\widehat B=\widehat A(\mu I_n)^{-1}T_n+\widehat B^\perp,
\qquad
\widehat A^*\widehat B^\perp=0,
\]
and define $\widehat K$ by
$\mathcal V_1\otimes\widehat P_2=\widehat K\mathcal V_3$. The forward Gram
calculation from Theorem~\ref{thm:gram} now gives
\begin{equation}\label{eq:higher-forward-gram}
M_{\mu I_n}^{(n)}=\widehat B^{\perp *}\widehat B^\perp
+\widehat K^*\widehat K\geq0.
\end{equation}

Since $S^{(n)}_{\mu I_n}=\mu I_{d_3}$, and
\begin{equation}\label{eq:symmetric-scalar-residual}
M_{\mu I_n}^{(n)}
=(1-\mu)I_{d_3}+(1-\mu^{-1})T_n^*T_n.
\end{equation}
The rows of $T_n$ have disjoint supports and squared norm
\[
\sigma^2\left(1+\frac{n-1}{3}\right)
=\frac{n+2}{3}\sigma^2.
\]
Therefore
\begin{equation}\label{eq:symmetric-T-norm}
T_nT_n^*=\lambda I_n,
\qquad
\lambda=\frac{n+2}{3}\sigma^2.
\end{equation}
By $d_3>n$, the kernel of $T_n$ is nonzero. Evaluating
\eqref{eq:symmetric-scalar-residual} on this kernel shows that
$1-\mu\geq0$.

We claim that $M_{\mu I_n}^{(n)}\neq0$. Otherwise
\eqref{eq:symmetric-scalar-residual}, evaluated on $\ker T_n$, would give
$\mu=1$. Equality in \eqref{eq:higher-forward-gram} would also give
$\widehat P_2=0$ and $\widehat B^\perp=0$. It would follow that
\[
\widehat P_3=\widehat A T_n\mathcal V_3
=q\widehat A\mathcal V_1=q\widehat P_1,
\]
and hence $\widehat P=g\widehat P_1$, contrary to reducedness. This proves
the claim. In particular, $\mu\neq1$, so $0<\mu<1$. On the range of
$T_n^*$, the eigenvalue in \eqref{eq:symmetric-scalar-residual} is
\[
(1-\mu)+(1-\mu^{-1})\lambda
=\frac{(1-\mu)(\mu-\lambda)}{\mu}.
\]
Its nonnegativity gives $\lambda\leq \mu<1$. Hence
$\lambda<1$, which is equivalent to \eqref{eq:symmetric-sharp-bound}.

Conversely, suppose \eqref{eq:symmetric-sharp-bound} holds, so that
$\lambda<1$, and choose $\lambda<\mu<1$.
Equations \eqref{eq:symmetric-scalar-residual} and
\eqref{eq:symmetric-T-norm} then show that $M_{\mu I_n}^{(n)}>0$.
Lemma~\ref{lem:higher-strict-gram} produces the required reduced map.
\end{proof}

We can now compare the natural sufficient condition from
Theorem~\ref{thm:higher-sufficiency} with the sharp symmetric bound. On
the symmetric ray, the former becomes
$n\sqrt{1-\sigma^2}>n-1,$
whereas the sharp condition is \eqref{eq:symmetric-sharp-bound}.
For $n=2$, these bounds agree, recovering the sharp region from
Section~\ref{sec:feasibility}. For every $n>2$, however, there is a
nonempty interval between them.

\begin{corollary}\label{cor:square-root-not-necessary}
For $n>2$, the condition in
\eqref{eq:higher-sufficient-condition} is sufficient but not necessary.
More precisely, if
\[
\frac{\sqrt{2n-1}}{n}
\leq\sigma<\sqrt{\frac3{n+2}},
\]
then the symmetric denominator $1+\sigma\sum_{j=1}^nz_j^2$ has a
reduced degree-three
realization, but the square-root condition fails.
\end{corollary}

\begin{proof}
On the symmetric ray, the square-root condition is equivalent to
$\sigma^2<(2n-1)/n^2$. For $n>2$, note that
\[
\frac3{n+2}-\frac{2n-1}{n^2}
=\frac{(n-1)(n-2)}{n^2(n+2)}>0.
\]
The asserted interval is therefore nonempty, and the conclusion follows
from Theorem~\ref{thm:symmetric-higher}.
\end{proof}

Thus the square-root condition that exactly characterizes admissibility
in two variables continues to provide a sufficient condition in every
dimension, but it ceases to be sharp once $n>2$. This highlights a distinction
between the two-variable problem studied in the preceding sections and
its higher-dimensional analogue.

\end{document}